\documentclass[leqno,a4paper,11pt]{amsart_mod}

\usepackage[all,arrow, matrix, curve]{xy}  
\usepackage{amsthm}
\usepackage{amssymb}
\usepackage{bbm}						
\usepackage{stmaryrd}                  	

\usepackage{graphicx}				 
\usepackage{array}				 	
\usepackage{footmisc}                  
\usepackage{float}					

\usepackage[english]{babel}
\usepackage[utf8]{inputenc}
\usepackage[T1]{fontenc} 
\usepackage{lmodern}          

\usepackage{longtable}
\usepackage{booktabs}
\usepackage[numbers]{natbib}   

\usepackage{tikz}
 \usetikzlibrary{calc,arrows,decorations.pathmorphing,backgrounds,positioning,fit,petri,chains,scopes}

\usepackage{pifont}

\usepackage{needspace}

\usepackage{caption}

\usepackage{hyperref}

\newtheoremstyle{satz}	
{3pt}			
{3pt}			
{\itshape}			
{}			
{\bfseries\sffamily}	
{:}			
{.5em}			
{}			
\newtheoremstyle{defi}	
{3pt}			
{3pt}			
{}			
{}			
{\bfseries\sffamily}	
{:}			
{.5em}			
{}			
\renewcommand\proof[1][\proofname]{%
  \par
  \pushQED{\qed}%
  \normalfont \topsep 6pt plus 6pt\relax
  \trivlist
  \item[\hskip\labelsep
        \itshape #1:]\ignorespaces
}

\theoremstyle{satz}
\newtheorem*{thm-plain}{Theorem}

\newtheorem{thm}{Theorem}

\newtheorem{lem}[thm]{Lemma}

\newtheorem{prp}[thm]{Proposition}

\theoremstyle{defi}

\theoremstyle{remark}

\newtheorem*{rem}{Remark}

\newcommand{\res}{\textnormal{res}}

\newcommand{\C}{\mathbb{C}}
\newcommand{\N}{\mathbb{N}}

\renewcommand{\thesection}{\arabic{section}}
\renewcommand{\labelenumi}{\roman{enumi})}

\begin{document}
	\pagestyle{plain}             
	\author{Bruno Niemann}
\address{Mathematisches Institut\\%
Universität zu Köln\\%
Weyertal 86--90, 50931 Köln, Germany}
\email{bniemann@math.uni-koeln.de}

\title{Spherical affine cones in exceptional cases and related branching rules}
\begin{abstract}
Given a complex simply connected simple algebraic group $G$ of exceptional type
and a maximal parabolic subgroup $P \subset G$, we classify all triples
$(G,P,H)$ such that $H \subset G$ is a maximal reductive subgroup acting
spherically on $G/P$. In addition we derive branching rules for $\text{res}^G_H
(V^*_{k\omega_i})$, $k \in \N$, where $\omega_i$ is the fundamental weight
associated to $P$.

This is the first of two parts of a project to classify all such triples and
corresponding branching rules for all simply connected simple algebraic groups.
%
%
\end{abstract}

	\maketitle
	\tableofcontents
	\section{Introduction} Given a reductive algebraic group $G$, a reductive
subgroup $H$ and some irreducible $G$-module $V$, then $V$ is also a $H$-module
in a natural way. An obvious problem is to find branching rules that describe
the decomposition of the $H$-module $V$ into irreducible components.

We will deal with this problem in
the situation where $G$ is a complex simply connected simple algebraic group of
exceptional type. The subgroup structure of these groups has been studied in
great detail and we want to consider maximal reductive subgroups of $G$. The
maximal closed connected subgroups are listed in Theorem 1 of \cite{Sei91}.
These groups are either semisimple or parabolic. So the maximal reductive
subgroups are easily obtained by adding the Levi factors of the maximal
parabolic groups which are maximal reductive in $G$ to the list of maximal
semisimple subgroups. The modules $V$ that we consider are those having as
highest weights a multiple of a fundamental weight.

We will approach this problem by working with spherical varieties. We consider
the flag variety $G/P$ where $P$ is a maximal parabolic subgroup of~$G$. 
Of special interest to us are the
flag varieties of that form, that are $H$-spherical, i.e.\ they contain an open
orbit for a Borel subgroup of $H$. The property of being
spherical can also be described in a representation-theoretic way. Namely a
normal affine $G$-variety is spherical if and only if its coordinate ring is a
multiplicity-free $G$-module \cite{Vin78}. Let $\widehat Y$ denote the
affine cone over $G/P$. Then the flag variety is
$H$-spherical if and only if all restrictions of the homogeneous
components of the coordinate ring of $\widehat Y$ to $H$ are multiplicity-free. 
These homogeneous comonents are exactly the irreducible submodules of the
coordinate ring $\C[\widehat Y]$ and they are of shape $V_{k\omega_i}^*$. In the
case of sphericity we can derive branching rules for these modules.

So the content of this paper is twofold. We classify the
spherical $H$-varieties $G/P$ and furthermore we derive branching rules for the
 simple $G$-submodules  of the coordinate ring of the affine
cones in the spherical cases. The results are summarized in Table~1. A flag
variety $G/P$ is $H$-spherical if and only if the branching rules for the
corresponding modules $V$ are given in the table.

\section{Notation}
We work over the field of
complex numbers throughout the article. $G$  always denotes a simply connected
simple algebraic group of exceptional type. Within $G$ we choose a Borel
subgroup $B$, a maximal torus $T$ and thereby define a set $\{\alpha_1,\ldots,\alpha_r\}$ of
simple roots which are labeled according to Bourbaki-notation. The system of
roots of $G$ is denoted by $\Phi$, the system of positive roots of $G$ is
denoted by $\Phi^+$ and $(a_1,\ldots,a_r)$ stands for the root
$\sum_{i=1}^r a_i \alpha_i$. Further $X_\alpha$ denotes a non-trivial
element of the root space associated to $\alpha$. Let $\Lambda^+$ be the set of
dominant weights related to $B$ and $T$. The irreducible $G$-module of highest
weight $\lambda \in \Lambda^+$ is denoted by $V_\lambda$. The fundamental weights of $G$ are
$\omega_1,\ldots,\omega_r$ and $\omega_1^*,\ldots,\omega_r^*$ are the
fundamental weights such that $(V_{\omega_i})^*= V_{\omega_i^*}$, where
$(V_{\omega_i})^*$ is the dual of $V_{\omega_i}$. If we write $k \omega_i$, then
$k \in \N$.

Let  $H$ denote a reductive subgroup of $G$ with root system $\Phi_H$
and analogous to $G$ we use  the notation $(b_1,\ldots,b_s)_H:=
\sum_{i=1}^s b_i \beta_i$ where $\{\beta_1,\ldots,\beta_s\}$ is a set of
simple roots of $\Phi_H$ given by the Borel subgroup $B_H=B\cap H$. The
fundamental weights of $H$ are denoted by $\lambda_1,\ldots,\lambda_s$, if $H$
is semisimple. When $H$ is a Levi subgroup, $\lambda_1,\ldots,\lambda_s$ denote
the fundamental weights of the semisimple part of $H$.

Lastly $\mathfrak b$ denotes the Lie algebra of $B_L$, $\mathfrak u$ the
Lie algebra of $U_L$ the unipotent radical of $B_L$ and $\mathfrak h$ the Lie
algebra of the maximal torus $T$ of $B_L$.

\section{Main results and outline of proof} 
We will now summarize the results and give an outline of the proof.  
In this paper we will derive the branching rules stated in the following table.
Further we show that if $\res^G_H(V_{k\omega_i})$ is given in the table, then
$G/P_{\omega_i^*}$ is a spherical $H$-variety. 
Conversely, if a maximal reductive subgroup $H\subset G$ does not appear
in the table, then the varieties $G/P_{\omega_i}$ are not $H$-spherical.

Note that for the subgroups $D_5\times \C^* \subset E_6$ and $E_6 \times \C^*
\subset E_7$ the weight of the $\C^*$-action depends on the embedding of $\C^*$.
The embedding that we chose is given in the corresponding sections. 

\needspace{4cm}
{\footnotesize
\newcolumntype{C}[1]{>{\centering\arraybackslash}p{#1}}
\begin{longtable}{|c|c|c|C{22mm}l|}\caption{}\label{tab:main_results}\\
\hline
$G$ & $H$ & $\omega$ & $\res^G_H (V_{\omega})$ & \\\toprule
\endhead\hline
$G_2$
&
$A_2$
&
$k\omega_1$
&
$\bigoplus\limits_{a_1+a_2\leq k}$
&
$V_{a_1\lambda_1+a_2 \lambda_2}$\\
&
&
$k\omega_2$
&
$\bigoplus\limits_{a_1+a_2+a_3=k}$
&
$V_{(a_1+a_3)\lambda_1+(a_2 + a_3)\lambda_2}$\\\hline\hline
$F_4$
&
$B_4$
&
$k\omega_1$
&
$\bigoplus\limits_{a_1+a_2=k}$
&
$V_{a_1\lambda_2+a_2 \lambda_4}$\\
&
&
$k\omega_2$
&
$\bigoplus\limits_{a_1+\ldots+a_5=k}$
&
$V_{(a_1+a_2)\lambda_1+(a_3 +
a_4)\lambda_2 + (a_1+a_5)\lambda_3 + (a_2 + a_4)\lambda_4}$\\
&
&
$k\omega_3$
&
$\bigoplus\limits_{a_1+\ldots+a_5=k}$
&
$V_{(a_1+a_5)\lambda_1 + a_2 \lambda_2 + a_3 \lambda_3 + (a_4+a_5)
 \lambda_4}$\\
&
&
$k\omega_4$
&
$\bigoplus\limits_{a_1+a_2\leq k}$
&
$V_{a_1\lambda_2+a_2 \lambda_4}$\\\hline\hline
$E_6$
&
$A_5\!\!\times\!\! A_1$
&
$k\omega_1$
&
$\bigoplus\limits_{a_1+2a_2+a_3=k}$
&
$V_{a_1\lambda_2+a_2\lambda_4 + a_3 \lambda_5}\otimes V_{a_3\lambda_6}$\\
&
&
$k\omega_6$
&
$\bigoplus\limits_{a_1+2a_2+a_3=k}$
&
$V_{a_1\lambda_1+a_2\lambda_2 + a_3 \lambda_4}\otimes V_{a_1
\lambda_6}$\\\cline{2-5}
&
$F_4$
&
$k\omega_1$
&
$\bigoplus\limits_{a_1 \leq k}$
&
$V_{a_1 \lambda_4}$\\
&
&
$k\omega_2$
&
$\bigoplus\limits_{a_1+a_2=k}$
&
$V_{a_1\lambda_1+a_2 \lambda_4}$\\
&
&
$k\omega_3$
&
$\bigoplus\limits_{a_1 + a_2 + a_3 = k}$
&
$V_{a_1 \lambda_1 + a_2 \lambda_3+ a_3 \lambda_4}$\\
&
&
$k \omega_5$
&
$\bigoplus\limits_{a_1+a_2+a_3=k}$
&
$V_{a_1\lambda_1+a_2\lambda_3 + a_3 \lambda_4}$\\
&
&
$k\omega_6$
&
$\bigoplus\limits_{a_1 \leq k}$
&
$V_{a_1 \lambda_4}$\\\cline{2-5}
&
$C_4$
&
$k\omega_1$
&
$\bigoplus\limits_{a_1+2a_2+2a_3 = k}$
&
$V_{a_1 \lambda_2 +a_2
\lambda_4}$\\
&
&
$k\omega_6$
&
$\bigoplus\limits_{a_1+2a_2+2a_3 = k}$
&
$V_{a_1 \lambda_2 +a_2 \lambda_4}$\\\cline{2-5}
&
$D_5 \times \C^*$
&
$k\omega_1$
&
$\bigoplus\limits_{a_1+a_2 +a_3=k}$
&
$V_{a_1 \lambda_1 +a_2 \lambda_4} \otimes V_{-2a_1+a_2+4a_3}$\\
&
&
$k\omega_2$
&
$\bigoplus\limits_{a_1+a_2+a_3 +a_4 = k}$
&
$V_{a_1 \lambda_2 +a_2 \lambda_4 + a_3 \lambda_5} \otimes V_{-3a_2+3a_3}$\\
&
&
$k\omega_3$
&
$\bigoplus\limits_{a_1+\ldots+a_6 = k}$
&
\begin{minipage}{6cm}
$V_{(a_1+a_6) \lambda_1 +a_2 \lambda_2+ a_3 \lambda_3 + (a_4+a_6) \lambda_4 +
a_5 \lambda_5} \otimes$\\
$\phantom{VV \otimes}V_{2a_1-4a_2+2a_3+5a_4 - a_5-3a_6}$
\end{minipage}
\\
&
&
$k\omega_5$
&
$\bigoplus\limits_{a_1+\ldots+a_6 = k}$
&
\begin{minipage}{6cm}
$V_{(a_1+a_6) \lambda_1 +a_2 \lambda_2+ a_3 \lambda_3 + a_4 \lambda_4 +
(a_5+a_6) \lambda_5} \otimes$\\
$\phantom{VV \otimes}V_{-2a_1+4a_2-2a_3+a_4-5a_5+a_6}$
\end{minipage}
\\
&
&
$k\omega_6$
&
$\bigoplus\limits_{a_1+a_2 +a_3 = k}$
&
$V_{a_1 \lambda_1 +a_2 \lambda_4} \otimes 
V_{2a_1-a_2-4a_3}$\\\hline\hline
$E_7$
&
$A_7$
&
$k\omega_7$
&
$\bigoplus\limits_{\substack{2a_1+a_2+\\2a_3+a_4=k}}$
&
$V_{a_2\lambda_2 + a_3
\lambda_4 + a_4 \lambda_6}$\\\cline{2-5}
&
$E_6\times \C^*$
&
$k\omega_1$
&
$\bigoplus\limits_{a_1+ a_2 + a_3 \leq k}$
&
$V_{a_1 \lambda_1 + a_2 \lambda_2 + a_3 \lambda_6} \otimes V_{2a_1-2a_3}$\\
&
&
$k\omega_2$
&
$\bigoplus\limits_{\substack{a_1+a_2+a_3+2a_4+\\a_5+a_6+a_7 = k}}$
&
\begin{minipage}{5cm}
\vspace{0.5ex}
$V_{a_1 \lambda_1 +(a_2+a_7) \lambda_2 + a_3 \lambda_3 + a_4 \lambda_4 +a_5
\lambda_5 + a_6 \lambda_6} \otimes$\\
$\phantom{VV\otimes} V_{-a_1+3a_2+a_3-a_5-2a_6}$
\end{minipage} 
\\
&
&
$k\omega_7$
&
$\bigoplus\limits_{\substack{a_1+a_2+\\a_3+a_4 = k}}$
&
$V_{a_1 \lambda_1 +a_2 \lambda_6}\otimes V_{-a_1+a_2+3a_3-3a_4}$\\\cline{2-5}
&
$D_6\!\!\times\!\! A_1$
&
$k\omega_7$
&
$\bigoplus\limits_{a_1+ 2 a_2 + a_3= k}$
&
$V_{a_1 \lambda_1 + a_2 \lambda_2 + a_3 \lambda_6} \otimes V_{a_1
\lambda_7}$\\\hline
\end{longtable}
}

To obtain the previous table we shall adapt the proof of Proposition 4.4 in
\cite{FL10} by Feigin and Littelmann. But first we will introduce some
additional notation.

Let $P_i \supset B$ denote the maximal parabolic subgroup of $G$ associated
to the fundamental weight $\omega_i$. We shall  consider the natural action of
$H$ on the projective varieties $Y=G/P_i$. The affine cone over $Y$ is denoted by
$\widehat Y$ and the stabilizer of $\overline 1 \in G/P_i$ is denoted by
$H_{\overline 1}$. The group $H_{\overline 1}$ is a parabolic subgroup of $H$.
Its opposite parabolic subgroup in $H$ is denoted by $Q$. Furthermore let $Q^u$
be its unipotent radical and let $L$ be the Levi-subgroup $H_{\overline 1} \cap
Q$ with Borel subgroup $B_L$ defined by the simple roots of $H$ that appear in
$L$. If we consider the orbit $O= H. \overline 1 \simeq H/H_{\overline 1}$ with
normal bundle $\mathcal N$ having fiber $N$ at~$\overline 1$ then $N$ has the
structure of an $L$-module since $L \subset H_{\overline 1}$.

If no confusion can arise we will write $P$ instead of $P_i$ from now on.

The proof is divided into two parts. First we
will determine in which cases $Y$ is a spherical $H$-variety. This part of
the proof is conducted in four steps.

{\itshape Step 1:} We apply the Brion-Luna-Vust Local
Structure Theorem \cite{BLV86} to get the following proposition.
\begin{prp}
There exists a locally closed affine subvariety $Z \subset Y$ such that
$\overline 1 \in Z$, $Z$ is stable under the action of $L$, $Q^u.Z$ is open in
$Y$ and the canonical map $Q^u\times Z \rightarrow Q^u.Z$ is an isomorphism of
varieties.
\end{prp}

\begin{proof}
Note that since the Borel subgroup $B_H$ is a subgroup of $P$, it is
contained in the stabilizer $H_{\overline 1}$ of $\overline 1 \in Y$. Thus
$H_{\overline 1}$ is a parabolic subgroup of~$H$.

Now we can apply the Local Structure Theorem to this situation and obtain the
proposition.
\end{proof}

{\itshape Step 2:} We have the  following proposition.
\begin{prp}
The variety $Y$ is $H$-spherical if and only if $Z$ is a
spherical $L$-variety.
\end{prp}

\begin{proof}
Assume $Z$ is spherical, i.e.\ a Borel subgroup of $L$
has a dense orbit in $Z$. Let $B_L$ be the Borel subgroup $B_H \cap L
\subset L$ and let $B_{L}^-$ be the opposite Borel subgroup. Then
$B^-_H=Q^u B^-_{L}$ is a Borel subgroup of $H$. Let $z \in Z$ be an element
such that $B^-_{L}. z$ is dense in $Z$. Since $Q^u.Z$ is dense in $Y$, so is
$B^-_H.z = Q^u(B^-_L.z)$. Hence $Y$ is a spherical $H$-variety.

If on the other
hand $Y$ is $H$-spherical, then $B_H^-.y=Q^u(B^-_L).y$ is open in $Y$ for some
$y \in Y$. Since $Q^u.Z$ is open in $Y$ we can assume that $y \in Z$. Now if
$Q^u (B_L^-.y)$ is dense in $Y$ it follows that $B_L^-.y$ is dense in $Z$.
\end{proof}

{\itshape Step 3:} Now $N$ is isomorphic to the tangent space  $T_{\overline 1}
Z$ and thanks to Luna's Slice Theorem $Y$ is $H$-spherical if and only if $N$ is
$L$-spherical.

{\itshape Step 4:}
It remains to compute $N$ and to check in which cases it
is a spherical $L$-module. Note that we have 
\begin{equation*}
N \simeq 
(\text{Lie}\,G/\text{Lie}\,P_i) / 
(\text{Lie}H/\text{Lie}\,H_{\overline 1}).
\end{equation*}
So if $\Phi_H \subset \Phi$, then we can describe $N$ as the root spaces that occur in
$T_{\overline 1}Y=\text{Lie}G/\text{Lie}P_i$ but not in
$T_{\overline 1}(H/H_{\overline 1})$.
 These are
all the root spaces $\C X_\alpha$ such that $\alpha$ is negative and $\C
X_\alpha \not \subset \text{Lie}P_i$ as well as $\C X_\alpha \not \subset
\text{Lie} H$. 

\begin{rem}
There is an algorithm by F.\ Knop \citep[Thm.\ 3.3]{Kno97} to check whether a
given $L$-module is spherical. But in order for this paper to be self-contained
we compute an explicit $X \in N$ such that $B_L.X$ is a dense orbit in $N$ in
the spherical cases.
\end{rem}

The second part is to compute the restrictions of the $G$-modules
$V_{k\omega_i^*}$ to~$H$. It is well-known that
\begin{equation*}
\C[\widehat Y]= \bigoplus_{k \geq 0} V_{k\omega^\ast_i}
\end{equation*}
where $V_{k\omega^\ast_i}$ corresponds to the homogeneous functions of degree
$k$ on $\widehat Y$. In order to derive branching rules for $V_{k \omega_i^*}$
we need to determine the $U_H$-invariants of $V_{k\omega^\ast_i}$.

Because $\widehat Y$ is a spherical $(H\times \C^*)$-variety and because
$U_H=U_{H\times \C^*}$, we know from Lemma 1 in \cite{Lit94} that the ring
$\C[\widehat Y]^{U_H}$ is a polynomial ring with some set of generators
$f_j$ of degree $d_j$, $1\leq j \leq s$, where $s$ is the number of generators.
Thus we have the following branching rules in this situation.
\begin{thm}
Let $\eta_j$ denote the weight of $f_j$ with respect to $H$ and suppose $G/P_i$
is a spherical $H$-variety. Then we get
\begin{equation*}
\text{res}^{G}_{H}(V_{k\omega^\ast_i}) = \bigoplus_{a_1 d_1 + \ldots +a_s d_s =
k} V_{a_1\eta_1+\ldots+a_s \eta_s}.
\end{equation*}
\end{thm}

We need to compute the number of generators, i.e.\ the dimension of
$\C[\widehat Y]^{U_H}$.
\begin{prp}\label{prp:codim}
We have
\begin{equation*}
\dim \C [\widehat Y]^{U_H}= \dim N -\dim (\textnormal{generic } U_L
\textnormal{-orbit}) +1.
\end{equation*}
\end{prp}
\begin{proof}
 We know that $\dim \C [\widehat Y]^{U_H}=
\text{trdeg}\;\C(\widehat Y)^{U_H}$ and by a theorem of Rosenthal we know that 
$\textnormal{trdeg}\,\, \C(\widehat Y)^{U_H}=\dim\widehat Y - \dim(\textnormal{generic
}U_H\textnormal{-orbit})$ (paragraph II.4.3.E in \citep[p.\ 143]{Kra84}).

So the proposition is an immediate corollary of the following lemma.
\end{proof}

\begin{lem}
Let $Y$, $N$, $U_L$ and $U_H$ be defined as above. Let $O_1$ be
a generic $U_H$-orbit in $Y$ and $O_2$ be a generic $U_L$-orbit in $N$.
Then 
\begin{equation*}
\dim Y - \dim O_1 = \dim N - \dim O_2.
\end{equation*}
\end{lem}
\begin{proof}
Let $O\subset Y$ be the open subset of $X$ such that $\dim U_H.x$ is maximal for
all $x \in O$ (i.e.\ $U_H.x$ is an generic orbit). We have $O \cap Q^u.Z \ne 
\emptyset$, because $Q^u . Z$ is open and dense in~$Y$.

Let $x=qz$ be an element in $O \cap Q^u.Z$. We know that
$U_H=U_L.Q^u=Q^u.U_L$. 
So we have
$U_H.x=U_H.(qz)= U_L Q^u (qz)= U_L Q^u .z = U_H.z$ and we can assume that $U_H
.x$ is a generic $U_H$-orbit in $Y$ with $x \in Z$.

Suppose $y$ is an element of the stabilizer $(U_H)_x$ of $x$.
Then we have $y=q.u$ for some $q\in Q^u$, $u\in U_L$. 
 So it follows from the Local Structure Theorem that $q= \text{id}$ and $ux=x$. 
Thus we get $(U_H)_x=(U_L)_x$.

With $\dim Y = \dim Z + \dim Q^u$ (Local
Structure Theorem) we get
\begin{equation*}
\begin{split}
\dim Y - \dim U_H.x 
&= \dim Q^u + \dim Z - \dim U_H.z\\
&= \dim Z -(\dim U_H.x -  \dim Q^u)\\
&= \dim Z -(\dim U_H- \dim (U_H)_x - \dim Q^u)\\
&= \dim Z -(\dim U_H- \dim Q^u - \dim (U_L)_x)\\
&= \dim Z -(\dim U_L - \dim (U_L)_x)\\
&= \dim Z -\dim U_L.x.
\end{split}
\end{equation*}
\end{proof}

\section{The maximal reductive subgroups of the exceptional groups}
We want to list all maximal reductive subgroups of the exceptional
algebraic groups. G.\ Seitz listed all maximal closed connected subgroups in
arbitrary characteristics. We recall his results for the case that the ground
field is $\C$ (\cite{Sei91}, Thm.\ 1).
\begin{thm}
Let $G$ be a simple algebraic group of exceptional type and let $X$ be maximal
among the proper closed connected subgroups of $G$. Then either $X$ contains a
maximal torus of $G$ or $X$ is semisimple and the pair $(G,X)$ is given below.
Moreover, maximal subgroups of each type exist and are unique up to conjugacy in
$\text{Aut}(G)$.

\begin{center}
\begin{tabular}{|c|c|c|}
\hline
$G$ & $X$ simple & $X$ not simple\\\hline\hline
$G_2$ & $A_1$ &\\\hline
$F_4$ & $A_1$ & $A_1\times G_2$\\ \hline
$E_6$ & $A_2$, $G_2$, $F_4$, $C_4$ & $A_2 \times G_2$\\ \hline
$E_7$ & $A_1$, $A_2$ & $A_1 \times A_1$, $A_1 \times G_2$, $A_1 \times F_4$,
$G_2 \times C_3$\\ \hline
$E_8$ & $A_1$, $B_2$ & $A_1 \times A_2$, $G_2 \times F_4$\\ \hline
\end{tabular}
\end{center}
\end{thm}

Since the maximal subgroups that do not contain a maximal torus are semisimple
they are also maximal reductive subgroups of $G$.

It remains to identify the maximal reductive subgroups that are contained in a
maximal subgroup of maximal rank. These groups fall in two categories. Some are
the maximal parabolic subgroups of $G$ and the others are so called subsystem
subgroups. There is an algorithm (cf.\ paragraph no.\ 17 of \cite{Dyn57.1} or \cite{BS49})
that determines these subgroups: Start with the Dynkin diagram of
$G$ and adjoin the smallest root $\delta$ to obtain the extended Dynkin diagram.
By removing a node from the extended diagram you arrive at the Dynkin diagram
of a subgroup of $G$. By Theorem\ 5.5 and the subsequent remark  in
\cite{Dyn57.1} these groups are maximal. Since they are semisimple they are also maximal
reductive.

To complete the list we need to consider the maximal parabolic subgroups of $G$.
Any reductive subgroup of a parabolic can be assumed to be a subgroup of its
Levi factor by Theorem\ 1 in \cite{LS96}. By considering the Dynkin diagrams it
is transparent that the Levi subgroups need not be maximal reductive but can be
subgroups of a subsystem subgroup. A simple case by case check shows that there
are only two Levi groups, that are maximal reductive.

Summarizing this we have the following maximal reductive subgroups containing a
maximal torus.
\begin{center}
\begin{tabular}{|c|c|c|}
\hline
$G$ & subsystem subgroups & Levi subgroups\\\hline\hline
$G_2$ & $A_2$, $A_1 \times A_1$ & \\ \hline
$F_4$ & $A_1 \times C_3$, $A_2 \times A_2$, $A_3 \times A_1$, $B_4$ &\\ \hline
$E_6$ & $A_5 \times A_1$, $A_2 \times A_2 \times A_2$ & $D_5 \times \C^*$\\
\hline $E_7$ & $D_6 \times A_1$ $A_5 \times A_2$, $A_3 \times A_3 \times A_1$,
$A_7$ & $E_6 \times \C^*$ \\\hline 
$E_8$ & $A_1 \times E_7$, $A_2 \times E_6$, $A_3 \times D_5$, $A_4
\times A_4$ &\\
      & $A_5 \times A_2 \times A_1$, $A_7 \times A_1$, $D_8$, $A_8$&\\ \hline
\end{tabular}
\end{center}

\section{\texorpdfstring{The exceptional group of type $G_2$}{The
exceptional group of type G2}}
We will now consider the simply connected simple algebraic group $G$ of type
$G_2$. The long roots of its root system form a subsystem of type $A_2$ and we
will consider the subsystem subgroup $H$ obtained in this way. The simple roots
of $H$ are given by
\begin{equation*}
(1,0)_{A_2}=(3,1) \text{ and } (0,1)_{A_2}=(0,1).
\end{equation*}

Using the same methods as before we can prove:
\begin{thm}
The varieties $G/P_1$ and $G/P_2$  are $H$-spherical.
\end{thm}
\begin{proof}\mbox{}\\
\noindent \underline{Case $G/P_1$:} We compute
\begin{equation*}
L=\langle T, U_{\pm (0,1)}\rangle.
\end{equation*}
and
\begin{equation*}
N= \C X_{-(1,0)_{G_2}} \oplus \C X_{-(1,1)_{G_2}} \oplus \C
X_{-(2,1)_{G_2}}.
\end{equation*}
If we define $X:= X_{-(1,1)} + X_{-(2,1)}$ we have $[\mathfrak b, X]=N$, which
shows that $N$ is $L$-spherical. It follows that $G/P_1$ is a spherical
$H$-variety.

\underline{Case $G/P_2$:} In this case we can compute that $L=T$ and 
\begin{equation*}
N = \C X_{-(1,1)} \oplus \C X_{-(2,1)}.
\end{equation*}
The module $N$ consists of two linearly independent root spaces and since $T$ is
2-dimensional $N$ is obviously $L$-spherical. That implies that $G/P_2$
is a spherical $H$-variety.
\end{proof}

\begin{thm}
Let $G$ be of type $G_2$ and $H$ of type $A_2$. Then we have the following
branching rules:
\begin{alignat*}{4}
\text{i)}\quad & &\res^G_H (V_{k \omega_1}) &= &
&\,\,\,\,\bigoplus_{a_1+a_2\leq k} &&V_{a_1\lambda_1+a_2
\lambda_2},\\
\text{ii)}\quad & &\res^G_H (V_{k \omega_2}) &= &
&\bigoplus_{a_1+a_2+a_3=k} &&V_{(a_1+a_3)\lambda_1+(a_2 + a_3)\lambda_2}.
\end{alignat*}
\end{thm}
\begin{rem}
In $G_2$ the fundamental weights are self-dual.
\end{rem}
\begin{proof}
\underline{i)} We use ``LiE'' to compute the restriction of $V_{\omega_1}$ and
get
\begin{equation*}
\res^G_H(V_{\omega_1})= \C \oplus V_{\lambda_1} \oplus V_{\lambda_2}.
\end{equation*}
Let $f_0,f_1,f_2$ be highest weight vectors of these representations. We
need to show that $\C[\widehat Y]^{U_H}$ is generated by these elements, i.e.\
we need to show that the dimension of  $\C[\widehat Y]^{U_H}$ is~3.

By considering $X_{-(1,0)}\in N$ we immediately see that the $U_L$-orbit of this
element is of codimension~2. Thus $\dim\C[\widehat Y]^{U_H}=3$ and since we
have already found three algebraically independent elements the branching rules
follow immediately.

\underline{ii)} We use ``LiE'' to compute
\begin{equation*}
\res^G_H(V_{\omega_2})= V_{\lambda_1} \oplus V_{\lambda_2} \oplus
V_{\lambda_1+\lambda_2}.
\end{equation*}
Let $f_1,f_2,f_3$ be highest weight vectors of these modules. We know that $U_L$
is the maximal torus in this case and so the unipotent radical is just the
identity. A generic orbit in $N$ is of dimension~0. And since $N$ is
2-dimensional, its codimension is~2. That means a generic $U_H$-orbit has
codimension 3 in $\widehat Y$ and that is also the dimension of $\C[\widehat
Y]^{U_H}$. We have already found three linearly independent elements which form
a generating set. The branching rules follow immediately.
\end{proof}

\begin{prp}
The varieties $G/P_i$ are not spherical $H$-varieties if $H$ is any other
maximal reductive subgroup of $G_2$.
\end{prp}
\begin{proof}
We have the following maximal reductive subgroups besides $A_2$: $A_1\times
A_1$ and $A_1$. If we compute the dimensions of a Borel subgroup in each case
and the dimensions of $G/P_i$ we obtain:
\begin{equation*}
\begin{array}{l|cc}
&  G/P_1 & G/P_2 \\\hline
\dim & 5 & 5
\end{array}
\end{equation*}
and 
\begin{equation*}
\begin{array}{l|cc}
H &  A_1\times A_1 & A_1\\\hline
 \dim B_H & 4 & 2
\end{array}
\end{equation*}
So $\dim B_H < \dim G/P_i$, $i=1,2$  for these subgroups.
\end{proof}

\section{\texorpdfstring{The exceptional group of type $F_4$}{The
exceptional group of type F4}}
In this section let $G$ be the group of type $F_4$.

Let $H$ be the subgroup of type $B_4$ in $G$. This is a subsystem subgroup so
from the Dynkin-diagram of $F_4$ we pass on to the extended Dynkin-diagram by
adding the smallest root $\delta$ to the system of simple roots.
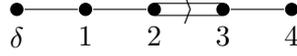
\begin{center}
\begin{tikzpicture}[vertex/.style={circle,fill,thick, inner sep=0pt,minimum
size=5pt}, node distance=3.5mm and 7 mm]
\node[vertex, label= below:$\delta$] (delta){};
\node[vertex, right =of  delta, label= below:$1$] (1){};
\node[vertex, right =of 1,label= below: $2$] (2) {};
\node[vertex, right =of 2,label= below: $3$] (3) {};
\node[vertex, right =of 3,label= below: $4$] (4) {};
\node at ($(2.east)!.5!(3.west)$) {$\rangle$};
\path
(delta.east) edge (1.west);
\path
(1.east) edge (2.west);

\path
($(2.east)+(-0.9mm,0.8mm)$) edge
($(3.west)+(0.9mm,0.8mm)$);
\path
($(2.east)+(-0.9mm,-0.8mm)$) edge
($(3.west)+(0.9mm,-0.8mm)$);
\path
(3.east) edge (4.west);
\end{tikzpicture}
\end{center}

By removing the simple root $\alpha_4$ we obtain a root-subsystem of type
$B_4$ and thus we find the corresponding subgroup $H\subset G$.

Explicitly we can choose the roots
\begin{alignat*}{2}
(1,0,0,0)_{B_4} &= (0,1,2,2), & \quad (0,1,0,0)_{B_4} &= (1,0,0,0),\\
(0,0,1,0)_{B_4} &= (0,1,0,0), & \quad (0,0,0,1)_{B_4} &= (0,0,1,0),\\
\end{alignat*}
which form a set of simple roots of a root subsystem of type $B_4$ in $F_4$.

We have the following theorem:
\begin{thm}
The varieties $G/P_i$, $i=1,\ldots,4$, are spherical $H$-varie\-ties. 
\end{thm}
\begin{proof}
We need to check that $N$ is a spherical $L$-module in each case.

\underline{Case $G/P_1$:} 
In this case we have
\begin{align*}
L =\langle T,
	& U_{\pm(0,1,2,2)}, U_{\pm(0,1,0,0)}, U_{\pm(0,0,1,0)},\\
	& U_{\pm(0,1,1,0)}, U_{\pm(0,1,2,0)}\rangle
\end{align*} 
and 
\begin{equation*}
N=\C X_{-(1,2,3,1)} \oplus \C X_{-(1,2,2,1)} \oplus
\C X_{-(1,1,2,1)} \oplus \C X_{-(1,1,1,1)}. 
\end{equation*}

The Borel subgroup $B_{L}$ of $L$ obviously contains the maximal torus $T$
of $G$. Since $N$ consists of four root spaces with linearly independent roots
and $T$ is 4-dimensional we know that there is a dense $B_{L}$-orbit in $N$.
Hence  $N$ is $L$-spherical and that implies that $G/P_1$ is $H$-spherical.

\underline{Case $G/P_2$:} 
Here we have
\begin{align*}
L=\langle T, U_{\pm(1,0,0,0)}, U_{\pm(0,0,1,0)} \rangle.
\end{align*}
We compute $N$ in the same way as in the previous case and get
\begin{align*}
N=\; &\C X_{-(0,1,1,1)} \oplus \C X_{-(0,1,2,1)} \oplus
   \C X_{-(1,1,1,1)} \oplus \C X_{-(1,1,2,1)} \oplus\\
   &\C X_{-(1,2,2,1)} \oplus \C X_{-(1,2,3,1)}. 
\end{align*}
We check the sphericity on the level of Lie algebras. Consider the element
\begin{equation*}
X:=X_{-(1,1,2,1)} + X_{-(0,1,2,1)} +
X_{-(1,1,1,1)} + X_{-(1,2,3,1)}
\end{equation*}
in $N$.
Then $[\mathfrak b, X]= N$. That means that $N$ is a
spherical $L$-variety and therefore $G/P_2$ is a spherical  $H$-variety.

\underline{Case $G/P_3$:} 
We get
\begin{align*}
N = & \C X_{-(0,0,1,1)} \oplus \C X_{-(0,1,1,1)} \oplus \C
	X_{-(0,1,2,1)} \oplus\\
	& \C X_{-(1,1,1,1)} \oplus \C X_{-(1,1,2,1)} \oplus \C
	X_{-(1,2,2,1)} \oplus \\
	& \C X_{-(1,2,3,1)}.
\end{align*}
If we consider 
\begin{equation*}
X:=X_{-(1,2,3,1)}+ X_{-(1,2,2,1)} +
X_{-(1,1,1,1)} + X_{-(0,1,2,1)} \in N
\end{equation*}
we have that $[\mathfrak b,X]=N$,
i.e.\ $N$ is a spherical $L$-variety and that means that $G/P_3$ is  a
spherical $H$-variety.

\underline{Case $G/P_4$:}
In this case we have
\begin{align*}
L=\langle T, 
		&U_{\pm(1,0,0,0)}, U_{\pm(0,1,0,0)}, U_{\pm(0,0,1,0)},\\
		&U_{\pm(1,1,0,0)}, U_{\pm(0,1,1,0)}, U_{\pm(1,1,1,0)},\\
		&U_{\pm(0,1,2,0)}, U_{\pm(1,1,2,0)},
		U_{\pm(1,2,2,0)}\rangle
\end{align*}
and
\begin{align*}
N= & \C X_{-(0,0,0,1)} \oplus \C X_{-(0,0,1,1)} \oplus \C
X_{-(0,1,1,1)} \oplus\\
 & \C X_{-(0,1,2,1)} \oplus \C X_{-(1,1,1,1)} \oplus \C
 X_{-(1,1,2,1)} \oplus\\
 & \C X_{-(1,2,2,1)} \oplus \C X_{-(1,2,3,1)}.
\end{align*}
The module $N$ has the following  structure.
\begin{equation*}
\xymatrix {
& & & X_{-(0,1,2,1)}\ar[dr]^{(0,0,1,0)}\\
X_{-(1,2,3,1)} \ar[r]^{(0,0,1,0)} & X_{-(1,2,2,1)} \ar[r]^{(0,1,0,0)} &
X_{-(1,1,2,1)} \ar[dr]_{(0,0,1,0)} \ar[ur]^{(1,0,0,0)} & &
X_{-(0,1,1,1)} \ar@{-}[r]& \cdots\\
& & & X_{-(1,1,1,1)} \ar[ur]_{(1,0,0,0)} & &\\
  \cdots \ar[r]^{(0,1,0,0)} & X_{-(0,0,1,1)} \ar[r]^{(0,0,1,0)}&
 X_{-(0,0,0,1)}& & }
\end{equation*}
We have $L=\C^* \times SO_7$ and $N$ is an irreducible $L$-module of
dimension~8. There exists only one such module which is the
$\text{Spin}_7$-module. That $N$ is a spherical $L$-module was proven by Victor
Kac \citep[][Thm. 3, p.~208]{Kac80}. It follows that $G/P_4$ is a spherical 
$H$-module.
\end{proof}

The spherical cases imply the following branching rules.
\begin{thm}
Let $G$ be of type $F_4$ and $H$ of type $B_4$. Then we have the following branching rules:
\begin{alignat*}{4}
\text{i)}\quad & &\textnormal{res}^G_H (V_{k \omega_1}) &= &
&\,\,\,\,\bigoplus_{a_1+a_2=k} &&V_{a_1\lambda_2+a_2
\lambda_4},\\
\text{ii)}\quad & &\textnormal{res}^G_H (V_{k \omega_2}) &= &
&\bigoplus_{a_1+\ldots+a_5=k} &&V_{(a_1+a_2)\lambda_1+(a_3 + a_4)\lambda_2 + (a_1+a_5)\lambda_3 + (a_2 +
a_4)\lambda_4},\\
\text{iii)}\quad & &\textnormal{res}^G_H (V_{k \omega_3}) &=
 &&\bigoplus_{a_1+\ldots+a_5=k} &&V_{(a_1+a_5)\lambda_1 + a_2 \lambda_2 + a_3 \lambda_3 + (a_4+a_5)
 \lambda_4},\\
\text{iv)}\quad & &\textnormal{res}^G_H (V_{k \omega_4}) &= &&
\,\,\,\,\bigoplus_{a_1+a_2\leq k} &&V_{a_1\lambda_2+a_2 \lambda_4}.\\
\end{alignat*}
\end{thm}

\begin{rem}
In $F_4$ the fundamental weights are self-dual.
\end{rem}

\begin{proof}\mbox{}\\
\underline{i):} 
Standard computations yield
\begin{equation*}
\textnormal{res}^G_H (V_{\omega_1}) = V_{\lambda_2} \oplus V_{\lambda_4}.
\end{equation*}
Let now $f_1,f_2 \in V_{\omega_1}$ be highest weight vectors of $V_{\lambda_2}$
and $V_{\lambda_4}$ respectively. We will show that $\C[\widehat Y]^{U_H}$
is generated by these degree 1 elements. We know that $\C[\widehat Y]^{U_H}$
is a polynomial ring. The grading  and weights of $f_1$ and $f_2$ imply that
they are algebraically independent. To rule out the possibility that
there are generators of degree two or higher we need to show that the Krull
dimension  of $\C[\widehat Y]^{U_H}$ is $2$.

Thus we need to find a generic $U_L$-orbit in $N$ and compute its
codimension. Since we have found 2 algebraically independent elements in
$\C[\widehat Y]^{U_H}$, we already know that the codimension must be at least
2.

Consider the Lie algebra $\mathfrak l$  of $L$. From above we know that the
Lie algebra $\mathfrak u$ of
$U_L$, is
\begin{equation*}
\mathfrak u = \C X_{(0,1,2,2)}\oplus \C X_{(0,1,0,0)} \oplus 
\C X_{(0,0,1,0)} \oplus \C X_{(0,1,1,0)} \oplus \C
X_{(0,1,2,0)}.
\end{equation*}

Define $X:= X_{-(1,2,3,1)} \in N$. Then
\begin{align*}
[X_{(0,1,2,2)},X]&= 0, & [X_{(0,1,0,0)},X]&=0, \\
[X_{(0,0,1,0)},X]&=X_{-(1,2,2,1)},
& [X_{(0,1,1,0)},X]&=X_{-(1,1,2,1)},\\
[X_{(0,1,2,0)},X]&=
X_{-(1,1,1,1)},
\end{align*}
which shows that the orbit of $X$ is of dimension 3. Thus a generic orbit has
dimension at least 3 with codimension at most~1. By Proposition \ref{prp:codim}
we know that in this case $\dim \C [\widehat Y]^{U_H}\leq 2$. But since we
have found two generators the dimension is exactly~2 and the restriction rules
follow.

\underline{ii):}
In this case we need to find generators of $\C[\widehat
Y]^{U_H}$. One can use the software ``LiE'' to compute
\begin{equation*}
\textnormal{res}^G_H (V_{\omega_2}) = V_{\lambda_1 + \lambda_3} \oplus
V_{\lambda_1 + \lambda_4} \oplus V_{\lambda_2} \oplus V_{\lambda_2 + \lambda_4}
\oplus V_{\lambda_3}.
\end{equation*}
Let $f_1,\ldots,f_5$ be highest weight vectors of these irreducible modules.

Consider $X:=X_{-(1,1,2,1)}+X_{-(1,2,3,1)} \in N$ and let $\mathfrak u$ be the
Lie-algebra of $U_L$ the unipotent radical of $L$. The stabilizer of this
element is just 0, which means that the dimension of a generic $U_L$-orbit is~2
with codimension~4. This implies that the codimension of a generic $U_H$-orbit
in $\widehat Y$ is 5. Thus $\C[\widehat Y]^{U_H}$ is generated by its degree 1
elements and the assertion follows. 

\underline{iii):}
We need to find generators of $\C[\widehat Y]^{U_H}$. One can use 
``LiE'' to compute
\begin{equation*}
\textnormal{res}^G_H (V_{\omega_3}) = V_{\lambda_1} \oplus 
V_{\lambda_2} \oplus V_{\lambda_3} \oplus V_{\lambda_4} \oplus V_{\lambda_1 +
\lambda_4}.
\end{equation*}
Let $f_1,\ldots,f_5$ be highest weight vectors of these irreducible modules.

Consider $X:=X_{-(1,1,1,1)}+X_{-(1,2,2,1)} \in N$ and take an element $u \in
\mathfrak u$ with $u = a X_{(1,0,0,0)} + bX_{(0,1,0,0)} + c
X_{(1,1,0,0)}$. Then
\begin{alignat*}{2}
&&\quad  [u,X]&=0\\
&\Rightarrow \quad & &= aX_{-(0,1,1,1)} + b X_{-(1,1,2,1)} + c
(X_{-(0,1,2,1)}+X_{-(0,0,1,1)})\\
&\Rightarrow & &a=b=c=0 \Rightarrow u=0
\end{alignat*}
and hence a generic $U_L$-orbit has dimension 3 with codimension 4. That means
that  $\C[\widehat Y]^{U_H}$ is of dimension 5 and generated by the elements
$f_i$.

\underline{iv):}
In this case we need to find generators of $\C[\widehat
Y]^{U_H}$. We use ``LiE''  to compute
\begin{equation*}
\textnormal{res}^G_H (V_{\omega_4}) = \C \oplus V_{\lambda_1} \oplus
V_{\lambda_4}.
\end{equation*}
Let $f_1,\ldots,f_3$ be highest weight vectors of these irreducible modules.

Consider $X:= X_{-(1,2,3,1)}$. We know that for 
\begin{equation*}X_{(1,0,0,0)}, X_{(0,1,0,0)},X_{(1,1,0,0)}
\in \mathfrak u
\end{equation*}
we have 
\begin{align*}
[X_{(1,0,0,0)},X]&=[X_{(0,1,0,0)},X]=[X_{(1,1,0,0)},X]=0\\
\end{align*}
and
\begin{align*}
[X_{(0,0,1,0)},X]&=X_{-(1,2,2,1)}, &
[X_{(0,1,1,0)},X]&=X_{-(1,1,2,1)},\\
[X_{(0,1,2,0)},X]&=X_{-(1,1,1,1)},&
[X_{(1,1,1,0)},X]&=X_{-(0,1,2,1)},\\
[X_{(1,1,2,0)},X]&=X_{-(0,1,1,1)},&
[X_{(1,2,2,0)},X]&=X_{-(0,0,1,1)}\\
\end{align*}
and thus the generic stabilizer is at most of dimension 3. The generic orbit is
at least of dimension 6 and thus its codimension is at most 2. This means that a
generic $U_H$-orbit in $\widehat Y$ is of dimension less or equal to 3.

Since we have found 3 algebraically independent elements the dimension of
$\C[\widehat Y]^{U_H}$ is exactly 3 and this finishes the proof.
\end{proof}

\begin{prp}
The varieties $G/P_i$ are not spherical $H$-varieties if $H$ is any other
maximal reductive subgroup of $F_4$.
\end{prp}
\begin{proof}
We have the following maximal reductive subgroups besides $B_4$: $A_1 \times
C_3$, $A_2 \times A_2$, $A_3 \times A_1$, $A_1\times G_2$ and $A_1$. If we
compute the dimensions of a Borel subgroup in each case and the dimensions of
$G/P_i$ we obtain:
\begin{equation*}
\begin{split}
&\begin{array}{l|cccc}
&  G/P_1 & G/P_2 & G/P_3 & G/P_4\\\hline
\dim \phantom{B_H}& 15 & 20 & 20 & 15
\end{array}\\
&\begin{array}{l|ccccc}
H &  A_1\times C_3 & A_2 \times A_2 & A_3 \times A_1 & A_1\times
G_2 & A_1\\\hline
 \dim B_H & 14 & 10 & 11 & 10 & 2
\end{array}
\end{split}
\end{equation*}
So we have $\dim B_H< \dim G/P_i$ for $i=1,\ldots,4$ in each case.
\end{proof}

\section{\texorpdfstring{The exceptional group of type $E_6$}{The exceptional
group of type E6}}
We will now turn to the group of type $E_6$. First we calculate the dimensions
of the Borel subgroups of the maximal reductive subgroups as well as the
dimensions of $G/P_i$ for $i=1,\ldots,6$.
\begin{align*}
&\begin{array}{l|c c c c c c c c}
H &A_5\times A_1& A_2\!\times\! A_2 \times\! A_2 & D_5 \times \C^*
& A_2 \times G_2 & G_2 & A_2 & F_4 & C_4\\\hline 
\dim B_H & 22 & 15 & 26 & 13 & 8 & 5 & 28 & 20 
\end{array}
\intertext{and}
&\begin{array}{l|cccccc}
& G/P_1 & G/P_2 & G/P_3 & G/P_4 & G/P_5 & G/P_6\\\hline
\dim \phantom{B_H}& 16 & 21 & 25 & 29 & 25 & 16
\end{array}\quad.
\end{align*}
Thus we get the following proposition.
\begin{prp}
Let $G$ be the simply connected simple algebraic group of  type
$E_6$ and let $H$ be a maximal reductive subgroup of type $A_2\times A_2 \times
A_2$, $A_2 \times G_2$, $G_2$ or $A_2$.\\
Then $G/P_i$ is not $H$-spherical for $i=1,\ldots,6$. 
\end{prp}
\begin{proof}
In these cases we have $\dim B_H < \dim G/P_i$ for $i=1,\ldots,6$. 
\end{proof}

Now we will consider the remaining groups and first we start with  the
subsystem subgroup of type $A_5 \times A_1$.
\begin{thm} Let $G$ be the simply connected simple algebraic group of 
type $E_6$ and let $H$ be the maximal reductive subgroup of type $A_5\times A_1$.
Then $G/P_1$ and $G/P_6$ are spherical $H$-varieties. The varieties
$G/P_2,\ldots,G/P_5$ are not $H$-spherical.
\end{thm}
\begin{proof}
The dimension of a Borel subgroup of a group of type $A_5\times A_1$ is~$22$.
Since we have $\dim G/P_3=25$, $\dim G/P_4=29$, $\dim G/P_5=25$ these varieties
cannot be spherical.

We know that $\omega^*_2=\omega_2$ in type $E_6$. Now if $G/P_2$ was a spherical
$H$-variety, $\text{res}^G_H(V_{k\omega_2})$ would be multiplicity-free for all
$k \in \N$ by what has been said above. But with ``LiE'' we compute
\begin{equation*}
\text{res}^G_H(V_{4\omega_2})=\ldots\oplus2 (V_{2\lambda_3}\otimes 
V_{3\lambda_6})\oplus\ldots
\end{equation*}
which means that there are multiplicities in this case.

To prove that $G/P_1$ and $G/P_6$ are spherical $H$-varieties we proceed as in
the cases above. We will show how $H$ is embedded in $G$. For doing so
we consider the extended Dynkin-diagram of type $E_6$ again by adding the
smallest root $\delta$ to the simple roots. Now omitting the root $\alpha_2$ we
obtain the embedding of $A_5 \times A_1$ in $E_6$.
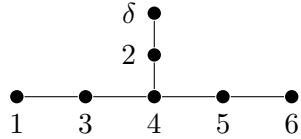
\begin{center}
\begin{tikzpicture}
{
[start chain, node
distance=3.5mm and 7 mm,vertex/.style={circle,fill,thick, inner sep=0pt,minimum
size=5pt}] 
\node[vertex,on chain,join,label= below: $1$] (1) {};
\node[vertex,on chain,join,label= below: $3$] (3) {};
\node[vertex,on chain,join,label= below: $4$] (4) {};
{[start branch=4]
\node[vertex,on chain=going above,join,label=left: $2$] (2) {};
\node[vertex,on chain=going above,join,label=left: $\delta$] (2) {};
}
\node[vertex,on chain,join,label= below: $5$] (5) {};
\node[vertex,on chain,join,label= below: $6$] (6) {};
}
\end{tikzpicture}
\end{center}
Explicitly we get the following set of simple roots:
\begin{alignat*}{2}
(1,0,0,0,0,0)_{A_5\times A_1} &= (1,0,0,0,0,0) & \quad 
(0,1,0,0,0,0)_{A_5\times A_1} &= (0,0,1,0,0,0)\\
(0,0,1,0,0,0)_{A_5\times A_1} &= (0,0,0,1,0,0) & \quad 
(0,0,0,1,0,0)_{A_5\times A_1} &= (0,0,0,0,1,0)\\
(0,0,0,0,1,0)_{A_5\times A_1} &= (0,0,0,0,0,1) & \quad 
(0,0,0,0,0,1)_{A_5\times A_1} &= (1,2,2,3,2,1)
\end{alignat*}
\underline{Case $G/P_1$:}
We compute
\begin{equation*}
\begin{split}
L =\langle &T, U_{\pm(0,0,1,0,0,0)}, U_{\pm(0,0,0,1,0,0)},
				 U_{\pm(0,0,0,0,1,0)}, U_{\pm(0,0,0,0,0,1)},\\
				 &U_{\pm(0,0,1,1,0,0)},U_{\pm(0,0,0,1,1,0)},
				 U_{\pm(0,0,0,0,1,1)},\\
				 &U_{\pm(0,0,1,1,1,0)}, U_{\pm(0,0,0,1,1,1)},
				 U_{\pm(0,0,1,1,1,1)}\rangle
\end{split}
\end{equation*}
and
\begin{equation*}
\begin{split}
N=\;&\C X_{-(1,1,1,1,0,0)} \oplus \C X_{-(1,1,1,1,1,0)} \oplus \C
X_{-(1,1,1,2,1,0)} \oplus \\
&\C X_{-(1,1,1,1,1,1)} \oplus \C X_{-(1,1,2,2,1,0)} \oplus \C
X_{-(1,1,1,2,1,1)} \oplus\\
& \C X_{-(1,1,2,2,1,1)} \oplus \C X_{-(1,1,1,2,2,1)} \oplus \C
X_{-(1,1,2,2,2,1)} \oplus\\
&\C X_{-(1,1,2,3,2,1)}.
\end{split}
\end{equation*}

Now let $X:= X_{-(1,1,2,3,2,1)}+X_{-(1,1,1,1,1,1)}$.
We have 
\begin{equation*}
[\mathfrak h, X]=
\langle
X_{-(1,1,2,3,2,1)},\, X_{-(1,1,1,1,1,1)}
\rangle,
\end{equation*}
since the roots are linearly independent.
Next we compute
{\allowdisplaybreaks
 \begin{alignat*}{2} 
 [X_{(0,0,0,1,0,0)},X]&=X_{-(1,1,2,2,2,1)}\quad &
 [X_{(0,0,0,1,1,0)},X]&=X_{-(1,1,2,2,1,1)}\\
 [X_{(0,0,1,1,0,0)},X]&=X_{-(1,1,1,2,2,1)}\quad &
 [X_{(0,0,1,1,1,0)},X]&=X_{-(1,1,1,2,1,1)}\\
 [X_{(0,0,0,1,1,1)},X]&=X_{-(1,1,2,2,1,0)}\quad &
 [X_{(0,0,1,1,1,1)},X]&=X_{-(1,1,1,2,1,0)}\\
 [X_{(0,0,0,0,0,1)},X]&=X_{-(1,1,1,1,1,0)}\quad &
 [X_{(0,0,0,0,1,1)},X]&=X_{-(1,1,1,1,0,0)}\\
 \end{alignat*}
 }%
and these computations show that we have ten linearly independent vectors in
$[\mathfrak b,X] \Rightarrow [\mathfrak b,X]=N \Rightarrow$ $N$ is a spherical
$L$-module. Hence $G/P_1$ is a spherical $H$-variety.

\underline{Case $G/P_6$:} The $H$-sphericity of $G/P_6$ is an immediate
corollary of the following theorem which states that 
$\C[\widehat Y]$ is multiplicity free.
\end{proof}
\begin{thm}
Let $G$ be the simply connected simple algebraic group of type $E_6$ and let
$H\subset G$ be the maximal reductive subgroup of type $A_5 \times A_1$.

Then we have the following branching rules:
\begin{alignat*}{4}
\text{i)}\quad & &\res^G_H (V_{k \omega_1}) &= &
&\,\,\,\,\bigoplus_{a_1+2a_2+a_3=k} &&
V_{a_1\lambda_2+a_2\lambda_4 + a_3 \lambda_5}\otimes V_{a_3\lambda_6},\\
\text{ii)}\quad & &\res^G_H (V_{k \omega_6}) &= &
&\,\,\,\,\bigoplus_{a_1+2a_2+a_3=k} &&
V_{a_1\lambda_1+a_2\lambda_2 + a_3 \lambda_4} \otimes V_{a_1 \lambda_6}.
\end{alignat*}
\end{thm}

\begin{rem}
In $E_6$ we have $\omega^*_1=\omega_6$, $\omega_2^*= \omega_2$,
$\omega_3^*=\omega_5$ and $\omega_4^*=\omega_4$.
\end{rem}
\begin{proof}
\underline{ii)} 	
With ``LiE'' we compute
\begin{equation*}
\begin{split}
\textnormal{res}^G_H(V_{\omega_6})&=(V_{\lambda_4} \otimes \C) \oplus
(V_{\lambda_1} \otimes V_{\lambda_6}),\\
\textnormal{res}^G_H(V_{2\omega_6})&=(V_{2\lambda_4}\otimes \C) \oplus
(V_{\lambda_1+\lambda_4}\otimes V_{\lambda_6})\oplus (V_{2\lambda_1}\otimes 
V_{2\lambda_6}) \oplus 
(V_{\lambda_2}\otimes \C).
\end{split}
\end{equation*}
There are at least two generators of degree~1 and of weights $(\lambda_4,0)$ and
$(\lambda_1,\lambda_6)$ and one generator of degree~2 and of weight
$(\lambda_2,0)$ for $\C [\widehat Y]^{U_H}$ with $Y=G/P_1$. In the proof of the
previous theorem we have found an element $X \in N$ with a $U_L$-orbit of codimension~2. So it
follows that $\dim \C[\widehat Y]^{U_H}=3$ and the branching rules
follow immediately.

\underline{i)} Theses branching rules follow directly from ii) by noting that
$\omega_1=\omega_6^*$, $\lambda_1^*=\lambda_5$, $\lambda_2^*=\lambda_4$ and
$\lambda_6^*=\lambda_6$.
\end{proof}

\begin{thm}\label{thm:E6-F4}
Let $G$ be the simply connected simple algebraic group of  type
$E_6$ and let $H$ be the maximal reductive subgroup of type $F_4$. 
Then $G/P_i$, $i \not =4$, are
spherical $H$-varieties. The variety $G/P_4$ is not
$H$-spherical.
\end{thm}
\begin{proof}
If we have the Dynkin diagrams 
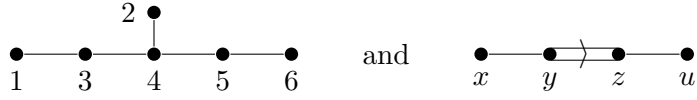
\begin{center}
\begin{tikzpicture}[vertex2/.style={circle,fill,thick, inner sep=0pt,minimum
size=5pt}, node distance=3.5mm and 7 mm]
{[start chain, node
distance=3.5mm and 7 mm,vertex/.style={circle,fill,thick, inner sep=0pt,minimum
size=5pt}] 
\node[vertex,on chain,join,label= below: $1$] (1) {};
\node[vertex,on chain,join,label= below: $3$] (3) {};
\node[vertex,on chain,join,label= below: $4$] (4) {};
{[start branch=4]
\node[vertex,on chain=going above,join,label=left: $2$] (2) {};
}
\node[vertex,on chain,join,label= below: $5$] (5) {};
\node[vertex,on chain,join,label= below: $6$] (6) {};
}
\node[right= of 6] (and) {and};
\node[vertex2, right= of and, label= below:$x$] (x){};
\node[vertex2, right =of x,label= below: $y$] (y) {};
\node[vertex2, right =of y,label= below: $z$] (z) {};
\node[vertex2, right =of z,label= below: $u$] (u) {};
\node at ($(y.east)!.5!(z.west)$) {$\rangle$};
\path
(x.east) edge (y.west);
\path
($(y.east)+(-0.9mm,0.8mm)$) edge
($(z.west)+(0.9mm,0.8mm)$);
\path
($(y.east)+(-0.9mm,-0.8mm)$) edge
($(z.west)+(0.9mm,-0.8mm)$);
\path
(z.east) edge (u.west);
\end{tikzpicture}
\end{center}
of $E_6$ and $F_4$, then we have an embedding of the simple Lie-algebra $F_4$ in
$E_6$ by choosing the following root vectors
\begin{alignat*}{2}
X_x&:=X_{(0,1,0,0,0,0)},\quad & X_z&:=\frac 1 {\sqrt 2}
(X_{(0,0,1,0,0,0)}+X_{(0,0,0,0,1,0)})\\
X_y &:= X_{(0,0,0,1,0,0)},& X_u&:=\frac 1 {\sqrt 2}
(X_{(1,0,0,0,0,0)}+X_{(0,0,0,0,0,1)})
\end{alignat*}
(\citep[p.\ 258, Table 24]{Dyn57.1} with different numbering of the Dynkin
diagrams). Now we consider the associated algebraic
subgroup of $E_6$.

\underline{Case $G/P_1$:}
We compute
\begin{equation*}
N= \C X_{-(1,1,1,2,2,1)}.
\end{equation*}
So $N$ is obviously $L$-spherical and thus $G/P_1$ is $H$-spherical.

\underline{Case $G/P_6$:} The $H$-sphericity of $Y=G/P_6$ is an immediate
corollary of the following theorem which states that 
$\C[\widehat Y]$ is multiplicity free.

\underline{Case $G/P_2$:} In this case we get
\begin{equation*}
\begin{split}
N= &\C X_{-(0,1,0,1,1,0)} \oplus   \C X_{-(0,1,0,1,1,1)} \oplus \C
X_{-(0,1,1,1,1,1)} \oplus\\ 
&\C X_{-(0,1,1,2,1,1)} \oplus \C X_{-(0,1,1,2,2,1)} \oplus \C
X_{-(1,1,1,2,2,1)_{E_6}}.
\end{split}
\end{equation*}
If we define $X:=X_{-(1,1,1,2,2,1)}$ then we have:
\begin{alignat*}{2}
[X_{(0,0,0,1)_{F_4}}, X]&= X_{-(0,1,1,2,2,1)}, &\quad
[X_{(0,0,1,1)_{F_4}},X]&=X_{-(0,1,1,2,1,1)},\\
[X_{(0,1,1,1)_{F_4}}, X]&= X_{-(0,1,1,1,1,1)}, &\quad
[X_{(0,1,2,1)_{F_4}},X]&=X_{-(0,1,0,1,1,1)},\\
[X_{(0,1,2,2)_{F_4}},X]&=X_{-(0,1,0,1,1,0)}.
\end{alignat*}
With $[\mathfrak h, X]=\C X$ we get $[\mathfrak b ,X]=N$ and it follows that $N$
is a spherical $L$-module.

\underline{Case $G/P_3$:} In this case we get
\begin{equation*}
\begin{split}
N=&\C X_{-(0,0,1,1,1,1)}\oplus \C X_{-(0,1,1,1,1,1)} \oplus \C
X_{-(0,1,1,2,1,1)} \oplus\\
 &\C X_{-(0,1,1,2,2,1)} \oplus \C X_{-(1,1,1,2,2,1)}.
\end{split}
\end{equation*}

Set $X:= X_{-(1,1,1,2,2,1)}+ X_{-(0,1,1,2,1,1)}$. Then
we have
\begin{equation*}
[\mathfrak h, X] = \C X_{-(1,1,1,2,2,1)} \oplus \C
X_{-(0,1,1,2,1,1)},
\end{equation*}  since the roots of the root vectors defining $X$ are linearly
independent. Furthermore we have
\begin{alignat*}{2}
[X_{(0,0,0,1)_{F_4}},X]&= X_{-(0,1,1,2,2,1)}, \quad &
[X_{(1,0,0,0)_{F_4}},X]&= X_{-(0,1,1,1,1,1)},\\
[X_{(1,1,0,0)_{F_4}},X]&= X_{-(0,0,1,1,1,1)}.
\end{alignat*}
So $[\mathfrak b, X]=N \Rightarrow$ $N$ is a spherical $L$-module and this
implies that $G/P_3$ is $H$-spherical.

\underline{Case $G/P_5$:} The $H$-sphericity of $Y=G/P_5$ is an immediate
corollary of the following theorem which states that 
$\C[\widehat Y]$ is multiplicity free.
\end{proof}

We can derive branching rules in the cases where $G/P_i$ is
a spherical $H$-variety.
\begin{thm}
Let $G$ be the simple simply connected algebraic group of type $E_6$  and $H$ be
the subgroup of type $F_4$.

Then we have the branching rules:
\begin{equation*}
\begin{array}{r r c c l}
\text{i)} &\res^G_H (V_{k \omega_1}) &\!\!\!=\!\!\! 
&\displaystyle\bigoplus\limits_{a_1 \leq k} & V_{a_1 \lambda_4},\\
\text{ii)} &\res^G_H (V_{k \omega_2})&\!\!\!=\!\!\!
&\displaystyle\bigoplus\limits_{a_1+a_2=k} & V_{a_1\lambda_1+a_2 \lambda_4},\\
\text{iii)} &\res^G_H (V_{k \omega_3}) &\!\!\!=\!\!\! 
&\displaystyle\bigoplus\limits_{a_1 + a_2 + a_3 = k} &
V_{a_1 \lambda_1 + a_2 \lambda_3+ a_3 \lambda_4},\\
\text{iv)} &\res^G_H (V_{k \omega_5})&\!\!\!=\!\!\! 
&\displaystyle\bigoplus\limits_{a_1+a_2+a_3=k} &
V_{a_1\lambda_1+a_2\lambda_3 + a_3 \lambda_4},\\
\text{v)} &\res^G_H (V_{k \omega_6}) &\!\!\!=\!\!\! 
&\displaystyle\bigoplus\limits_{a_1 \leq k} &
V_{a_1 \lambda_4}.
\end{array}
\end{equation*}
\end{thm}

\begin{proof}
\underline{v)} In this case we work with $Y=G/P_1$. With ``LiE'' we compute
\begin{equation*}
\res^G_H(V_{\omega_6})= \C \oplus V_{\lambda_4}.
\end{equation*}
Since $N$ is 1-dimensional in this case, each $U_L$-orbit is 0-dimensional
with codimension~1. So $\dim \C[\widehat Y]^{U_H} = 2$ and $\C [\widehat
Y]^{U_H}$ is generated by its degree-1-elements. The branching rules follow.

\underline{i)} Theses branching rules follow directly from v) by noting that
$\omega_1=\omega_6^*$ and $\lambda_i^*=\lambda_i$.

\underline{ii)} In this case we work with $Y=G/P_2$. With ``LiE'' we compute
\begin{equation*}
\res^G_H(V_{\omega_2})= V_{\lambda_1} \oplus V_{\lambda_4},
\end{equation*} 
so there are two generators of degree 1. The module $N$ is of dimension 6 and
we have seen that $X_{-(1,1,1,2,2,1)}\in N$ is an element such that $U_L.X$ is
of dimension 5. So $\dim \C[\widehat Y]^{U_H}\leq 2$ and hence $\C [\widehat
Y]^{U_H}$ is generated by its degree-1-elements. The branching rules follow immediately.

\underline{iv)} In this case we work with $G/P_3$. With ``LiE'' we compute
\begin{equation*}
\res^G_H(V_{\omega_5})= V_{\lambda_1} \oplus V_{\lambda_3} \oplus V_{\lambda_4},
\end{equation*} 
so again there are 3 generators of degree 1. The module $N$ is of
dimension~5 and  $X_{-(1,1,1,2,2,1)_{E_6}}+ X_{-(0,1,1,2,1,1)_{E_6}}$ is an
element of $N$ with a 3-dimensional $U_L$-orbit (cf.\ proof of previous
theorem). So $\dim \C [\widehat Y]^{U_H}\leq 3$. It follows that $\C [\widehat
Y]^{U_H}$ is generated by its degree-1-elements and so the branching rules follow.

\underline{iii)} These branching rules follow directly from v) by noting that
$\omega_3=\omega_5^*$ and $\lambda_i^*=\lambda_i$.
\end{proof}

\begin{thm}
Let $G$ be the simply connected simple algebraic group of  type
$E_6$ and let $H$ be the maximal reductive subgroup of type $C_4$. 
Then $G/P_1$ and $G/P_6$  are
spherical $H$-varieties. The varieties $G/P_2,\ldots,G/P_5$ are not
$H$-spherical.
\end{thm}

\begin{proof}
That $G/P_2,\ldots,G/P_5$ are not $H$-spherical follows by dimension reasons.

For the other two cases we consider the Dynkin diagrams
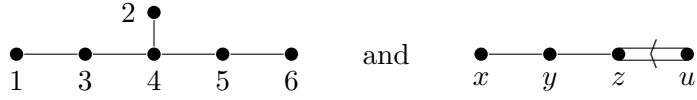
\begin{center}
\begin{tikzpicture}[vertex2/.style={circle,fill,thick, inner sep=0pt,minimum
size=5pt}, node distance=3.5mm and 7 mm]
{[start chain, node
distance=3.5mm and 7 mm,vertex/.style={circle,fill,thick, inner sep=0pt,minimum
size=5pt}] 
\node[vertex,on chain,join,label= below: $1$] (1) {};
\node[vertex,on chain,join,label= below: $3$] (3) {};
\node[vertex,on chain,join,label= below: $4$] (4) {};
{[start branch=4]
\node[vertex,on chain=going above,join,label=left: $2$] (2) {};
}
\node[vertex,on chain,join,label= below: $5$] (5) {};
\node[vertex,on chain,join,label= below: $6$] (6) {};
}
\node[right= of 6] (and) {and};
\node[vertex2, right= of and, label= below:$x$] (x){};
\node[vertex2, right =of x,label= below: $y$] (y) {};
\node[vertex2, right =of y,label= below: $z$] (z) {};
\node[vertex2, right =of z,label= below: $u$] (u) {};
\node at ($(z.east)!.5!(u.west)$) {$\langle$};
\path
(x.east) edge (y.west);
\path
(y.east) edge (z.west);
\path
($(z.east)+(-0.9mm,0.8mm)$) edge
($(u.west)+(0.9mm,0.8mm)$);
\path
($(z.east)+(-0.9mm,-0.8mm)$) edge
($(u.west)+(0.9mm,-0.8mm)$);
\end{tikzpicture}
\end{center}
of $E_6$ and $C_4$ respectively. Then the simple Lie-algebra of type $C_4$ is
embedded into the simple Lie-algebra of type $E_6$ by choosing the following
root vectors:
\begin{alignat*}{2}
X_x&\!:=\!\frac 1 {\sqrt 2} (X_{(0,1,1,1,0,0)}\!+\!X_{(0,1,0,1,1,0)}),\quad
& X_y&\!:=\!\frac 1 {\sqrt 2} (X_{(1,0,0,0,0,0)}+X_{(0,0,0,0,0,1)})\\
X_z\! &:=\!\frac 1 {\sqrt 2} (X_{(0,0,1,0,0,0)}\! +\! X_{-(0,0,0,0,1,0}),&
X_u&\!:=\!X_{(0,0,0,1,0,0)}
\end{alignat*}
(cf.\ \citep[p.\ 258, Table 24]{Dyn57.1}).
Now we consider the associated subgroup $H$ of $G$.

\underline{Case $G/P_1$:} We compute
\begin{equation*}
\begin{split}
N= &\C X_{-(1,1,1,1,1,1)}\oplus \C X_{-(1,1,1,2,1,1)} \oplus \C
X_{-(1,1,2,2,1,1)} \oplus\\
& \C X_{-(1,1,2,2,2,1)} \oplus \C X_{-(1,1,2,3,2,1)}.
\end{split}
\end{equation*}
We define $X:= X_{-(1,1,2,3,2,1)}+X_{-(1,1,1,1,1,1)}$.
Then we have 
\begin{equation*}
[\mathfrak h, X]= \C X_{-(1,1,2,3,2,1)}\oplus \C
X_{-(1,1,1,1,1,1)}.
\end{equation*}
Further we get
\begin{alignat*}{2}
&[X_{(0,0,0,1)_{C_4}},X]=X_{-(1,1,2,2,2,1)},\quad &
[X_{(0,0,1,1)_{C_4}},X]&=X_{-(1,1,2,2,1,1)}\\
& [X_{(0,0,2,1)_{C_4}},X]=X_{-(1,1,1,2,1,1)}.
\end{alignat*}
This implies that $[\mathfrak b, X]$ contains five linearly independent vectors
of $N$ $\Rightarrow [\mathfrak b, X]=N$. Hence  $N$ is $L$-spherical.

\underline{Case $G/P_6$:} The $H$-sphericity of $Y=G/P_6$ is an immediate
corollary of the following theorem which states that 
$\C[\widehat Y]$ is multiplicity free.
\end{proof}

From the spherical cases we can derive the following branching rules:
\begin{thm}
Let $G$ be the simply connected simple algebraic group of type $E_6$ and $H$ be
the subgroup of type $C_4$.

Then we have the following branching rules:
\begin{equation*}
\begin{array}{r r c c l}
\text{i)} &\res^G_H (V_{k \omega_1})&\!\!\!=\!\!\!
&\displaystyle\bigoplus\limits_{a_1+2a_2+2a_3 = k} & V_{a_1 \lambda_2 +a_2
\lambda_4},\\
\text{ii)} &\res^G_H (V_{k \omega_6}) &\!\!\!=\!\!\! 
&\displaystyle\bigoplus\limits_{a_1+2a_2+2a_3 = k} & V_{a_1 \lambda_2 +a_2
\lambda_4}.
\end{array}
\end{equation*}
\end{thm}
\begin{proof}\mbox{}\\
\underline{ii)} Here we are in the case $Y=G/P_1$. With ``LiE'' we compute
\begin{align*}
\res^G_H (V_{\omega_6}) = V_{\lambda_2} \quad \text{and}\quad
\res^G_H (V_{2\omega_6}) = \C \oplus V_{2\lambda_2} \oplus V_{\lambda_4}.
\end{align*}
So  there is one generator of degree 1 and two of degree 2 in $\C[\widehat
Y]^{U_H}$.
From the calculations in the proof of the previous
theorem we know that $X_{-(1,1,2,3,2,1)}+ X_{-(1,1,1,1,1,1)}$ is an element of
$N$ whose $U_L$-orbit is of codimension~2. Hence $\dim \C[\widehat Y]^{U_H}
\leq 3$. But since we have already found three generators we know that  $\dim
\C[\widehat Y]^{U_H} = 3$. The branching rules follow immediately.

\underline{i)} Theses branching rules follow directly from ii) by noting
that $\omega_1=\omega_6^*$, $\lambda_2^*=\lambda_2$ and  $\lambda_4^*=\lambda_4$.
\end{proof}

Next we will consider the Levi subgroup $H$ of $G$ that is obtained by omitting
the simple root $\alpha_1$. From the Dynkin diagram of $E_6$ we see that $H$
is the group $D_5 \times \C^*$.
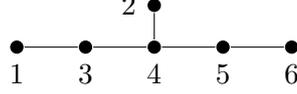
\begin{center}
\begin{tikzpicture}[vertex2/.style={circle,fill,thick, inner sep=0pt,minimum
size=5pt}, node distance=3.5mm and 7 mm]
{[start chain, node
distance=3.5mm and 7 mm,vertex/.style={circle,fill,thick, inner sep=0pt,minimum
size=5pt}] 
\node[vertex,on chain,join,label= below: $1$] (1) {};
\node[vertex,on chain,join,label= below: $3$] (3) {};
\node[vertex,on chain,join,label= below: $4$] (4) {};
{[start branch=4]
\node[vertex,on chain=going above,join,label=left: $2$] (2) {};
}
\node[vertex,on chain,join,label= below: $5$] (5) {};
\node[vertex,on chain,join,label= below: $6$] (6) {};
}
\end{tikzpicture}
\end{center}

\begin{thm}
Let $G$ be the simply connected simple algebraic group of  type
$E_6$ and let $H$ be the Levi subgroup  $D_5 \times \C^*$.\\
Then $G/P_i$ is a spherical $H$-variety for $i \ne 4$. The variety $G/P_4$ is
not $H$-spherical.
\end{thm}
\begin{proof}
This is proven in \cite{Lit94}.
\end{proof}

\begin{thm}
Let $G$ be the simply connected simple algebraic groups of type $E_6$ let
$H\subset G$ be the Levi subgroup  $D_5 \times \C^*$. Then we have the following
branching rules.
{\allowdisplaybreaks
\begin{equation*}
\begin{array}{r r c c l}
\text{i)} &\res^G_H (V_{k \omega_1})&\!\!\!=\!\!\!
&\displaystyle\bigoplus\limits_{a_1+a_2 +a_3= k} & 
V_{a_1 \lambda_1 + a_2 \lambda_4} \otimes V_{-2a_1+a_2+4a_3},\\
\text{ii)} &\res^G_H (V_{k \omega_2})&\!\!\!=\!\!\!
&\displaystyle\bigoplus\limits_{a_1+a_2+a_3 + a_4 = k} & 
V_{a_1 \lambda_2 + a_2 \lambda_4 + a_3 \lambda_5} \otimes V_{-3a_2+3a_3},\\
\text{iii)} &\res^G_H (V_{k \omega_3})&\!\!\!=\!\!\!
&\displaystyle\bigoplus\limits_{a_1+\ldots+a_6 = k} & 
\begin{minipage}{6.5cm}
$V_{(a_1+a_6) \lambda_1 +a_2 \lambda_2+ a_3 \lambda_3 + (a_4+a_6) \lambda_4 +
a_5 \lambda_5} \otimes$\\
$\phantom{VV \otimes}V_{2a_1-4a_2+2a_3+5a_4 - a_5-3a_6}$
\end{minipage}
,\\ 
\text{iv)} &\res^G_H (V_{k\omega_5})&\!\!\!=\!\!\! &
\displaystyle\bigoplus\limits_{a_1+\ldots+a_6 = k} &
\begin{minipage}{6.5cm}
$V_{(a_1+a_6) \lambda_1 +a_2 \lambda_2+ a_3 \lambda_3 + a_4+\lambda_4 +
(a_5+a_6) \lambda_5} \otimes$\\
$\phantom{VV \otimes}V_{-2a_1+4a_2-2a_3+a_4-5a_5+a_6}$
\end{minipage}
,\\
\text{v)} &\res^G_H (V_{k \omega_6}) &\!\!\!=\!\!\!
&\displaystyle\bigoplus\limits_{a_1+a_2 +a_3 = k} &
 V_{a_1 \lambda_1 + a_2 \lambda_5}\otimes V_{2a_1-a_2-4a_3}.
\end{array}
\end{equation*}
}
\end{thm}

\begin{proof} From paragraph 1.4 in \cite{Lit94} we get the following branching
rules.
 {\allowdisplaybreaks
\begin{equation*}
\begin{array}{r r c c l}
\text{i)} &\res^G_H (V_{k \omega_1})&\!\!\!=\!\!\!
&\displaystyle\bigoplus\limits_{a_1+a_2 +a_3= k} & V_{(a_3-a_1-a_2) \omega_1
+a_2 \omega_3 + a_1 \omega_6},\\
\text{ii)} &\res^G_H (V_{k \omega_2})&\!\!\!=\!\!\!
&\displaystyle\bigoplus\limits_{a_1+a_2+a_3 + a_4 = k} & V_{-(a_1+2a_2)
\omega_1 +a_3 \omega_2 + a_2 \omega_3 + a_1 \omega_5},\\
\text{iii)} &\res^G_H (V_{k \omega_3})&\!\!\!=\!\!\!
&\displaystyle\bigoplus\limits_{a_1+\ldots+a_6 = k} &
V_{\substack{-(2a_2+a_3+a_5+2a_6) \omega_1 +a_5 \omega_2+ (a_4+a_6) \omega_3\\ +
a_3\omega_4 + a_2 \omega_5+(a_1+a_6)\omega_6}},\\ 
\text{iv)} &\res^G_H (V_{k
\omega_5})&\!\!\!=\!\!\! &\displaystyle\bigoplus\limits_{a_1+\ldots+a_6 = k} &
V_{\substack{-(a_1+2a_3+a_4+2a_5+a_6) \omega_1 +(a_5+a_6)\omega_2+ a_4
\omega_3\\ +a_3 \omega_4 + a_2 \omega_5 +(a_1+a_6) \omega_6}},\\ 
\text{v)} &\res^G_H (V_{k \omega_6}) &\!\!\!=\!\!\!
&\displaystyle\bigoplus\limits_{a_1+a_2 +a_3 = k} & V_{-(a_2+a_3) \omega_1 +a_2
\omega_2 + a_1 \omega_6}.
\end{array}
\end{equation*}
}
We would like to write these highest weights in terms of the fundamental weights
of $D_5$ and $\C^*$. We have  $\omega_6= \lambda_1$, $\omega_5=\lambda_2$,
$\omega_4= \lambda_3$, $\omega_3=\lambda_4$ and $\omega_2=\lambda_5$, where $\lambda_i$ are the
fundamental weights of $D_5$ and we fix the coweight
$3\omega_1^\vee=4\alpha_1^\vee + 3\alpha_2^\vee + 5\alpha_3^\vee + 6
\alpha_4^\vee + 4 \alpha_5^\vee+ 2 \alpha_6^\vee$ which determines the highest
weights for $\C^*$. Thus we get the branching rules in the theorem. 
\end{proof}

\section{\texorpdfstring{The exceptional group of type $E_7$}{The exceptional
group of type E7}}
Let $G$ be of type $E_7$ with the following 
Dynkin-Diagram.
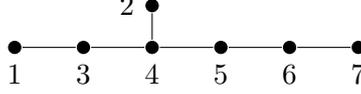
\begin{center}
\begin{tikzpicture}[vertex2/.style={circle,fill,thick, inner sep=0pt,minimum
size=5pt}, node distance=3.5mm and 7 mm]
{[start chain, node
distance=3.5mm and 7 mm,vertex/.style={circle,fill,thick, inner sep=0pt,minimum
size=5pt}] 
\node[vertex,on chain,join,label= below: $1$] (1) {};
\node[vertex,on chain,join,label= below: $3$] (3) {};
\node[vertex,on chain,join,label= below: $4$] (4) {};
{[start branch=4]
\node[vertex,on chain=going above,join,label=left: $2$] (2) {};
}
\node[vertex,on chain,join,label= below: $5$] (5) {};
\node[vertex,on chain,join,label= below: $6$] (6) {};
\node[vertex,on chain,join,label= below: $7$] (7) {};
}
\end{tikzpicture}
\end{center}
For this group there are only a few cases of sphericity as we will see. As we
did in the last section we start by calculating the dimensions of the
Borel subgroups of the maximal reductive subgroups as well as the dimensions of
$G/P_i$ for $i=1,\ldots,7$.

We have
\begin{equation*}
\begin{array}{l | c c c c c c c}
	& G/P_1 & G/P_2 & G/P_3 & G/P_4 & G/P_5 & G/P_6 & G/P_7\\\hline
\dim & 33 & 42 & 47& 53 & 50 &  42 &  27
\end{array}.
\end{equation*}
For the Borel subgroups $B_H$ we have:
\begin{equation*}
\begin{split}
&\begin{array}{l | c c c c c c}
H & A_7 & E_6\!\times \! \C^* & A_3 \times A_3 \times A_1 & A_5\times A_2 &
D_6\times A_1 & A_1 \times A_1\\\hline
\dim B_H & 35 & 43 & 20 & 25 & 38 & 4
\end{array}\\
&\begin{array}{l| c c c c c}
H&  A_1 \times G_2 & G_2 \times C_3 & A_1 \times F_4 & A_1 & A_2\\\hline
\dim B_H & 10 & 20 & 30 & 2 &5
\end{array}
\end{split}
\end{equation*}
So we can rule out a lot of cases by dimension comparison.
\begin{prp}
Let $G$ be the simply connected simple algebraic group of type $E_7$. If $H$ is
a maximal reductive subgroup of type $A_3 \times A_3 \times A_1$, $A_5\times
A_2$, $A_1\times A_1$, $A_1\times G_2$, $G_2\times C_3$, $A_1$ or $A_2$, then
$G/P_i$ is not a spherical $H$-variety for $i=1,\ldots,7$. \qed
\end{prp}
\begin{proof}
In these cases we have $\dim B_H < \dim G/P_i$ for $i=1,\ldots,7$. 
\end{proof}

Now we turn to the remaining subgroups and start with the subgroup of type
$A_7$. This is a subsystem subgroup so we add the smallest root $\delta$
to the simple roots and consider the extended Dynkin diagram.
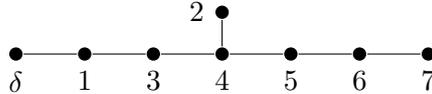
\begin{center}
\begin{tikzpicture}[vertex2/.style={circle,fill,thick, inner sep=0pt,minimum
size=5pt}, node distance=3.5mm and 7 mm]
{[start chain, node
distance=3.5mm and 7 mm,vertex/.style={circle,fill,thick, inner sep=0pt,minimum
size=5pt}]
\node[vertex,on chain,join,label= below: $\delta$] (delta) {}; 
\node[vertex,on chain,join,label= below: $1$] (1) {};
\node[vertex,on chain,join,label= below: $3$] (3) {};
\node[vertex,on chain,join,label= below: $4$] (4) {};
{[start branch=4]
\node[vertex,on chain=going above,join,label=left: $2$] (2) {};
}
\node[vertex,on chain,join,label= below: $5$] (5) {};
\node[vertex,on chain,join,label= below: $6$] (6) {};
\node[vertex,on chain,join,label= below: $7$] (7) {};
}
\end{tikzpicture}
\end{center}
By omitting the simple root $\alpha_2$ we obtain the embedding of the root
system $A_7$ into $E_7$. Explicitly we get
{\allowdisplaybreaks
\begin{alignat*}{2}
(1,0,0,0,0,0,0)_{A_7}\!&=(1,0,0,0,0,0,0),\;&
(0,1,0,0,0,0,0)_{A_7}\!&=(0,0,1,0,0,0,0),\\
(0,0,1,0,0,0,0)_{A_7}\!&=(0,0,0,1,0,0,0),&
(0,0,0,1,0,0,0)_{A_7}\!&=(0,0,0,0,1,0,0),\\
(0,0,0,0,1,0,0)_{A_7}\!&=(0,0,0,0,0,1,0),&
(0,0,0,0,0,1,0)_{A_7}\!&=(0,0,0,0,0,0,1),\\
(0,0,0,0,0,0,1)_{A_7}\!&=(1,2,2,3,2,1,0).&
\end{alignat*}
}
Now we consider the corresponding subsystem subgroup $H$.
\begin{thm}\label{thm:E7-A7}
Let $G$ be the simply connected simple algebraic group of type $E_7$ and $H$ the
maximal reductive subgroup of type $A_7$. Then $G/P_7$ is a spherical
$H$-variety whereas $G/P_i$ is not $H$-spherical for $i \ne 7$.
\end{thm}	
\begin{proof}
By dimension comparison $G/P_i$ can only be spherical for $i=1$ or $i=7$.
We know that for $E_7$ we have $\omega_i^*=\omega_i$. And with LiE we compute
\begin{equation*}
\text{res}^G_H(V_{4\omega_1})=\ldots \oplus 2 V_{\lambda_4}\oplus\ldots\,.
\end{equation*}
This shows that we have multiplicities in this case and $G/P_1$ is not a
spherical $H$-variety.

For $G/P_7$ we use the same methods as above. We compute
\begin{equation*}
\begin{split}
N
=
&\C X_{-(0,1,0,1,1,1,1)} \oplus \C X_{-(0,1,1,1,1,1,1)} \oplus
\C X_{-(1,1,1,1,1,1,1)} \oplus\\
&\C X_{-(0,1,1,2,1,1,1)} \oplus \C X_{-(1,1,1,2,1,1,1)} \oplus \C
X_{-(0,1,1,2,2,1,1)} \oplus\\
&\C X_{-(1,1,2,2,1,1,1)} \oplus \C X_{-(1,1,1,2,2,1,1)} \oplus \C
X_{-(0,1,1,2,2,2,1)} \oplus\\
&\C X_{-(1,1,2,2,2,1,1)} \oplus \C X_{-(1,1,1,2,2,2,1)} \oplus \C
X_{-(1,1,2,3,2,1,1)} \oplus\\
&\C X_{-(1,1,2,2,2,2,1)} \oplus \C X_{-(1,1,2,3,2,2,1)} \oplus \C
X_{-(1,1,2,3,3,2,1)}.
\end{split}
\end{equation*}

Define $X:= X_{-(1,1,2,3,3,2,1)} + X_{-(1,1,2,2,1,1,1)}+ X_{-(0,1,0,1,1,1,1)}$.
The roots of the root-vectors in $X$ are linearly independent. Thus we get that
\begin{equation*}
[\mathfrak h, X] := \langle X_{-(1,1,2,3,3,2,1)_{E_7}}, X_{-(1,1,2,2,1,1,1)_{E_7}},
X_{-(0,1,0,1,1,1,1)_{E_7}}\rangle
\end{equation*}
and further
\begin{alignat*}{2}
[X_{(0,0,1,0,0,0,0)},X]&= X_{-(1,1,1,2,1,1,1)}, & \quad
[X_{(0,0,0,0,1,0,0)},X]&= X_{-(1,1,2,3,2,2,1)},\\
[X_{(1,0,1,0,0,0,0)},X]&= X_{-(0,1,1,2,1,1,1)},&
[X_{(0,0,1,1,0,0,0)},X]&= X_{-(1,1,1,1,1,1,1)},\\
[X_{(0,0,0,1,1,0,0)},X]&= X_{-(1,1,2,2,2,2,1)},&
[X_{(0,0,0,0,1,1,0)},X]&= X_{-(1,1,2,3,2,1,1)},\\
[X_{(1,0,1,1,0,0,0)},X]&= X_{-(0,1,1,1,1,1,1)},&
[X_{(0,0,1,1,1,0,0)},X]&= X_{-(1,1,1,2,2,2,1)},\\
[X_{(0,0,0,1,1,1,0)},X]&= X_{-(1,1,2,2,2,1,1)},&
[X_{(1,0,1,1,1,0,0)},X]&= X_{-(0,1,1,2,2,2,1)},\\
[X_{(0,0,1,1,1,1,0)},X]&= X_{-(1,1,1,2,2,1,1)},&
[X_{(1,0,1,1,1,1,0)},X]&= X_{-(0,1,1,2,2,1,1)}.
\end{alignat*}
This shows that $\dim [\mathfrak b,X]=15 = \dim N \Rightarrow $
$[\mathfrak b,X]=N$ $\Rightarrow$ $N$ is a spherical $L$-module.
And thus  $G/P_7$ is a spherical $H$-variety.
\end{proof}
Since $G/P_7$ is a spherical $H$-variety we can derive branching rules for
$V_{k\omega^*_7}=V_{k\omega_7}$.

\begin{thm}
Let $G$ be the simply connected simple algebraic group of type $E_7$ and $H$
the maximal reductive  subgroup of type $A_7$. Then
\begin{equation*}
\res^G_H(V_{k\omega_7})= \bigoplus_{2a_1+a_2+2a_3+a_4=k} V_{a_2\lambda_2 + a_3
\lambda_4 + a_4 \lambda_6}.
\end{equation*}
\end{thm}

\begin{proof} 
With ``LiE'' we compute
\begin{equation*}
\res^G_H (V_{\omega_7})= V_{\lambda_2} \oplus V_{\lambda_6}.
\end{equation*}
So there are two generators of degree 1  of weight $\lambda_2$ and
$\lambda_6$. Further we have
\begin{equation*}
\res^G_H (V_{2\omega_7})=\C \oplus V_{2\lambda_2} \oplus V_{2\lambda_6} \oplus
V_{\lambda_2 + \lambda_6} \oplus V_{\lambda_4},
\end{equation*}
which shows that there are 2 generators of degree 2 which are of weight $0$ and
$\lambda_4$. This shows that $\dim \C [\widehat Y]^{U_H}\geq 4$.

In the proof of the previous theorem we have found an $X \in N$
such that $U_L.X$ is of  codimension~3. It follows that $\dim \C
[\widehat Y]^{U_H}=4$ and we have found four generators. The branching rules
follow immediately.
\end{proof}

Next we will consider the Levi subgroup $E_6 \times \C^*$, which is
obtained by omitting the simple root $\alpha_7$ in the
Dynkin-diagram.

\begin{thm}
Let $G$ be the simply connected simple algebraic group of type $E_7$ and $H
\subset G$ the Levi subgroup of type $E_6 \times \C^*$. Then  $G/P_1$ and
$G/P_7$ are spher\-ical $H$-varieties whereas $G/P_i$, $i=2,\ldots,6$ are not
spherical $H$-varieties.
\end{thm}

\begin{proof} This was proven in \cite{Lit94}.
\end{proof}
We get the following branching rules from the spherical cases.
\begin{thm}
Let $G$ be the simply connected simple algebraic group of type $E_7$ and $H$
the Levi subgroup of type $E_6\times \C^*$. Then we have the following branching
rules. 
{\allowdisplaybreaks
\begin{equation*}
\begin{array}{r r c c l}
\text{i)} &\res^G_H (V_{k \omega_1})&\!\!\!=\!\!\!
&\displaystyle\bigoplus\limits_{a_1+ a_2 + a_3 +a_4 = k}
&V_{a_1 \lambda_1+ a_2 \lambda_2+ a_3 \lambda_6}\otimes V_{2a_1-2a_3},\\
\text{ii)} &\res^G_H (V_{k \omega_2})&\!\!\!=\!\!\!
&\displaystyle\bigoplus\limits_{\substack{a_1+a_2+a_3+2a_4+\\a_5+a_6+a_7 = k}} &
\begin{minipage}{6.2cm}
$V_{a_1 \lambda_1 +(a_2+a_7) \lambda_2 + a_3 \lambda_3 + a_4 \lambda_4+ a_5
\lambda_5 + a_6 \lambda_6} \otimes$\\
 $\phantom{VV\otimes} V_{-a_1+3a_2+a_3-a_5-2a_6}$
 \end{minipage}
 ,\\
\text{iii)} &\res^G_H (V_{k \omega_7})&\!\!\!=\!\!\!
&\displaystyle\bigoplus\limits_{a_1+a_2+a_3+a_4 = k} & V_{a_1 \lambda_1 + a_2
\lambda_6} \otimes V_{-a_1+a_2+3a_3-3a_4}.\\
\end{array}
\end{equation*}
}
\end{thm}
\begin{proof} From paragraph 1.4 in \cite{Lit94} we get the following branching
rules.
{\allowdisplaybreaks
\begin{equation*}
\begin{array}{r r c c l}
\text{i)} &\res^G_H (V_{k \omega_1})&\!\!\!=\!\!\!
&\displaystyle\bigoplus\limits_{a_1+ a_2 + a_3 +a_4 = k}
&V_{a_1 \omega_1+ a_2 \omega_2+ a_3 \omega_6 -(a_2+2a_3)\omega_7},\\
\text{ii)} &\res^G_H (V_{k \omega_2})&\!\!\!=\!\!\!
&\displaystyle\bigoplus\limits_{\substack{a_1+a_2+a_3+2a_4+\\a_5+a_6+a_7 = k}} &
V_{\substack{a_1 \omega_1 +(a_2+a_7) \omega_2 + a_3 \omega_3 + a_4 \omega_4+ a_5
\omega_5 + a_6 \omega_6\\ \quad -(a_1+a_3+2a_4+2a_5+2a_6+a_7)\omega_7}}\;\;,\\
\text{iii)} &\res^G_H (V_{k \omega_7})&\!\!\!=\!\!\!
&\displaystyle\bigoplus\limits_{a_1+a_2+a_3+a_4 = k} & V_{a_1 \omega_1 + a_2
\omega_6+(a_3-a_1-a_2-a_4)\omega_7}.\\
\end{array}
\end{equation*}
}
We have $\omega_i = \lambda_i$ for $i=1,\ldots,6$ and we fix the coweight
$2\omega_7^\vee=2\alpha_1^\vee + 3 \alpha_2^\vee + 4 \alpha_3^\vee + 6
\alpha_4^\vee + 5 \alpha_5^\vee + 4 \alpha_6^\vee + 3 \alpha_7^\vee$ which
determines the highest weights for $\C^*$. Thus we get the branching rules
in the theorem.
\end{proof}

Now we will turn to the subgroup of $E_7$ of type $D_6\times A_1$. We will
consider the extended Dynkin-diagram of $E_7$ again by adding the smallest root
$\delta$ to the simple roots. 
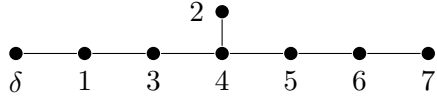
\begin{center}
\begin{tikzpicture}[vertex2/.style={circle,fill,thick, inner sep=0pt,minimum
size=5pt}, node distance=3.5mm and 7 mm]
{[start chain, node
distance=3.5mm and 7 mm,vertex/.style={circle,fill,thick, inner sep=0pt,minimum
size=5pt}]
\node[vertex,on chain,join,label= below: $\delta$] (delta) {}; 
\node[vertex,on chain,join,label= below: $1$] (1) {};
\node[vertex,on chain,join,label= below: $3$] (3) {};
\node[vertex,on chain,join,label= below: $4$] (4) {};
{[start branch=4]
\node[vertex,on chain=going above,join,label=left: $2$] (2) {};
}
\node[vertex,on chain,join,label= below: $5$] (5) {};
\node[vertex,on chain,join,label= below: $6$] (6) {};
\node[vertex,on chain,join,label= below: $7$] (7) {};
}
\end{tikzpicture}
\end{center}
If we omit the simple root $\alpha_6$ we have a sub-diagram of type $D_6\times
A_1$ and consider the the corresponding subsystem subgroup.
Explicitly we can choose the following simple roots:
{\allowdisplaybreaks
\begin{alignat*}{2}
(1,0,0,0,0,0,0)_H&= (0,1,1,2,2,2,1),& \;\,
(0,1,0,0,0,0,0)_H&= (1,0,0,0,0,0,0),\\
(0,0,1,0,0,0,0)_H&= (0,0,1,0,0,0,0),&
(0,0,0,1,0,0,0)_H&= (0,0,0,1,0,0,0),\\
(0,0,0,0,1,0,0)_H&= (0,1,0,0,0,0,0),&
(0,0,0,0,0,1,0)_H&= (0,0,0,0,1,0,0),\\
(0,0,0,0,0,0,1)_H&= (0,0,0,0,0,0,1).
\end{alignat*}
}

\begin{thm}
Let $G$ be the simply connected simple algebraic group of type $E_7$. If $H$ is
the subgroup of type $D_6\times A_1$ then $G/P_7$ is a spherical $H$-variety and
$G/P_i$ is not a spherical $H$-variety for $i=1,\ldots,6$.
\end{thm}
\begin{proof}
Dimension comparison shows that $G/P_2,\ldots,G/P_6$ are not $H$-spher\-ical.
For $G/P_1$ we can compute the restriction of $V_{k\omega_1}$ (note that
$\omega^*_i=\omega_i$ for $E_7$) with LiE and get
\begin{equation*}
\text{res}^G_H(V_{4\omega_1})=\ldots\oplus 2(V_{2\lambda_6} \otimes
V_{2\lambda_7})\oplus\ldots\;\;.
\end{equation*}
Thus there are multiplicities in this case and we know that the $H$-variety
$G/P_1$ is not $H$-spherical.

\underline{Case $G/P_7$:} We compute
{\allowdisplaybreaks
\begin{align*}
\quad N 
=
&\C X_{-(0,0,0,0,0,1,1)} \oplus \C X_{-(0,0,0,0,1,1,1)} \oplus \C
X_{-(0,0,0,1,1,1,1)}\oplus\\
&\C X_{-(0,1,0,1,1,1,1)} \oplus \C X_{-(0,0,1,1,1,1,1)} \oplus \C
X_{-(1,0,1,1,1,1,1)} \oplus\\
&\C X_{-(0,1,1,1,1,1,1)}\oplus \C X_{-(1,1,1,1,1,1,1)} \oplus \C
X_{-(0,1,1,2,1,1,1)} \oplus\\
&\C X_{-(1,1,1,2,1,1,1)} \oplus \C X_{-(0,1,1,2,2,1,1)}\oplus \C
X_{-(1,1,2,2,1,1,1)} \oplus \\
&\C X_{-(1,1,1,2,2,1,1)} \oplus \C X_{-(1,1,2,2,2,1,1)} \oplus \C
X_{-(1,1,2,3,2,1,1)} \oplus\\
&\C X_{-(1,2,2,3,2,1,1)}.
\end{align*}
}
Now define $X:= X_{-(1,2,2,3,2,1,1)}+ X_{-(1,0,1,1,1,1,1)}$. The roots of these
two root vectors are linearly independent and we have
\begin{equation*}
[\mathfrak h, X]=\langle X_{-(1,2,2,3,2,1,1)}, X_{-(1,0,1,1,1,1,1)}\rangle
\end{equation*}
Further we have
{\allowdisplaybreaks
\begin{alignat*}{2}
[X_{(1,0,0,0,0,0,0)},X]&= X_{-(0,0,1,1,1,1,1)}, \;\; &
[X_{(0,1,0,0,0,0,0)},X]&= X_{-(1,1,2,3,2,1,1)},\\
[X_{(1,0,1,0,0,0,0)},X]&= X_{-(0,0,0,1,1,1,1)},&
[X_{(0,1,0,1,0,0,0)},X]&= X_{-(1,1,2,2,2,1,1)},\\
[X_{(1,0,1,1,0,0,0)},X]&= X_{-(0,0,0,0,1,1,1)},&
[X_{(0,1,1,1,0,0,0)},X]&= X_{-(1,1,1,2,2,1,1)},\\
[X_{(0,1,0,1,1,0,0)},X]&= X_{-(1,1,2,2,1,1,1)},&
[X_{(1,1,1,1,0,0,0)},X]&= X_{-(0,1,1,2,2,1,1)},\\
[X_{(1,0,1,1,1,0,0)},X]&= X_{-(0,0,0,0,0,1,1)},&
[X_{(0,1,1,1,1,0,0)},X]&= X_{-(1,1,1,2,1,1,1)},\\
[X_{(1,1,1,1,1,0,0)},X]&= X_{-(0,1,1,2,1,1,1)},&
[X_{(0,1,1,2,1,0,0)},X]&= X_{-(1,1,1,1,1,1,1)},\\
[X_{(1,1,1,2,1,0,0)},X]&= X_{-(0,1,1,1,1,1,1)},&
[X_{(1,1,2,2,1,0,0)},X]&= X_{-(0,1,0,1,1,1,1)}.
\end{alignat*}
}
So we have $\dim [\mathfrak b,X]=16=\dim N$. This implies that $N$ is a
spherical $L$-module and thus $G/P_7$ is a spherical $H$-variety.
\end{proof}
From the sphericity of $G/P_7$ we can derive branching rules for $V_{k
\omega^*_7}= V_{k \omega_7}$.

\begin{thm}
Let $G$ be the simply connected simple algebraic group of type $E_7$ and let $H$
be a maximal reductive subgroup of type $D_6 \times A_1$.
Then
\begin{equation*}
\res^G_H (V_{k \omega_7}) = \bigoplus_{a_1+ 2 a_2 + a_3= k}
 V_{a_1 \lambda_1 + a_2 \lambda_2 + a_3 \lambda_6} \otimes V_{a_1 \lambda_7}.
\end{equation*}
\end{thm}
\begin{proof}
With ``LiE'' we compute
\begin{equation*}
\res^G_H (V_{\omega_7})= (V_{\lambda_1} \otimes V_{\lambda_7}) \oplus
(V_{\lambda_6} \otimes \C).
\end{equation*}
So there are two generators of degree 1 with weights $(\lambda_1, \lambda_7)$
and $(\lambda_6,0)$.
Further we have
\begin{equation*}
\res^G_H (V_{2\omega_7})= (V_{2\lambda_1}\otimes V_{2\lambda_7}) \oplus
(V_{\lambda_1 + \lambda_6}\otimes V_{\lambda_7}) \oplus (V_{2\lambda_6}\otimes
\C) \oplus (V_{\lambda_2}\otimes \C).
\end{equation*}
Thus there is a further generator of degree 2 and weight $\lambda_2$ and we know
that $\dim \C [\widehat Y]^{U_H}\geq 3$.

In the proof of the previous theorem we have seen that there is an $X \in N$
such that $\dim U_H.X$ is of codimension 2. 
It follows that $\dim \C [\widehat Y]^{U_H}= 3$. The branching  rules follow.
\end{proof}

The last maximal reductive subgroup of $G$ where a sphericity of $G/P_i$ can
occur is the group $H$ of type $A_1 \times F_4$. From the table with dimensions of
$G/P_i$ we know that only $G/P_7$ can be a spherical $H$-variety. But with LiE
we compute
\begin{equation*}
\text{res}^G_H(V_{4\omega_7})= \ldots \oplus 2 (V_{4\lambda_1}\otimes
V_{\lambda_5})
\oplus
\ldots
\end{equation*}
and thus there are multiplicities in this case. We have shown:
\begin{thm}
Let $G$ be the simply connected simple group of type $E_7$ and $H$ the maximal
subgroup of type $A_1\times F_4$.

Then $G/P_i$ ($i=1,\ldots,7$) is not a spherical variety.\qed 
\end{thm}

\section{\texorpdfstring{The exceptional group of type $E_8$}{The exceptional
group of type E8}}
We start our computations again by calculating the dimensions of the Borel
subgroups of the maximal reductive subgroups and the dimensions of $G/P_i$ for
$i=1,\ldots,8$.
\begin{equation*}
\begin{split}
&\begin{array}{l| c c c c c c c c c}
H
&E_7\!\!\times\!\! A_1
&E_6 \!\! \times\!\! A_2
&A_3 \!\! \times\!\! D_5
&A_4 \!\! \times\!\! A_4
&A_5 \!\! \times\!\! A_2 \!\! \times\!\! A_1\\ \hline
\dim B_H
&72
&47
&34
&28
&27
\end{array}\\
&
\begin{array}{l| ccccccc}
H
&A_7 \!\! \times \!\! A_1
& D_8
& A_8
& G_2 \!\! \times \!\! F_4
& A_2 \!\! \times \!\! A_1
& C_2
& A_1\\ \hline
\dim B_H
& 37
& 72
& 44
& 36
& 6
& 6
& 2
\end{array}
\end{split}
\end{equation*}
The dimensions of the varieties $G/P_i$ ($i=1,\ldots,8$) are: 
\begin{equation*}
\begin{array}{l|c c c c c c c c}
&G/P_1
&G/P_2
&G/P_3
&G/P_4
&G/P_5
&G/P_6
&G/P_7
&G/P_8\\\hline
\dim
& 78
& 92
& 98
& 106
& 104
& 97
& 83
& 57
\end{array}
\end{equation*}

By dimension comparison there are only two possibilities of sphericity. If we
take the maximal reductive subgroup $H_1$ of type $E_7 \times A_1$ or the
maximal reductive subgroup $H_2$ of type $D_8$, then the variety $G/P_8$ can be
spherical for $H_1$ or $H_2$. But we can compute the following restrictions by using LiE
\begin{equation*}
\begin{split}
\text{res}^G_{H_1}(V_{5\omega_8})&= \ldots \oplus  2 (V_{1\lambda_1 +
2\lambda_7} \otimes V_{2 \lambda_8}) \oplus \ldots,\\
 \text{res}^G_{H_2}(V_{4\omega_8})&=\ldots \oplus 2 V_{\lambda_8} \oplus 
 \ldots,\\
\end{split}
\end{equation*}
which show that there are multiplicities in these cases. So there are no
spherical cases for~$G$. We have shown:
\begin{thm}
Let $G$ be the simply connected simple algebraic groups of type~$E_8$. Let $H$
be one of its maximal reductive subgroups.

Then $G/P_i$ ($i=1,\ldots,8$) is not a spherical variety.\qed
\end{thm}

	\bibliographystyle{amsalpha}
	\bibliography{literatur}
	
\end{document}